\documentclass{article}

\usepackage{amsmath,amssymb,amsfonts,amsthm,amssymb}

\newcommand{\aff}{\mathop{\rm aff}}

\def\B{{\cal B}}
\def\C{{\cal C}}

\def\F{{\cal F}}

\def\L{{\cal L}}
\def\M{{\cal M}}
\def\N{{\cal N}}
\def\P{{\cal P}}

\newcommand{\midb}{\;\middle|\;}

\def\comp{\mathop{\text{\scriptsize $\circ$}}}

\def\minimize{\mathop{\rm minimize}\limits}

\def\st{\mathop{\rm subject\ to}}

\def\dom{\mathop{\rm dom}}
\def\cl{\mathop{\rm cl}}

\def\reals{\mathbb{R}}
\def\inte{\mathop{\rm int}}

\def\ovr{\mathop{\rm over}}
\def\tos{\rightrightarrows}
\def\rinte{\mathop{\rm rint}}

\newtheorem{assumption}{Assumption}
\newtheorem{theorem}{Theorem}
\newtheorem{lemma}[theorem]{Lemma}
\newtheorem{corollary}[theorem]{Corollary}

\theoremstyle{definition}

\newtheorem{remark}[theorem]{Remark}

\title{Convex duality in nonlinear optimal transport}
\author{Teemu Pennanen\thanks{Department of Mathematics, King's College London,
Strand, London, WC2R 2LS, United Kingdom} \and 
Ari-Pekka Perkki\"o\thanks{Mathematics Institute, Ludwig-Maximilian University of Munich, Theresienstr. 39, 80333 Munich, Germany}}

\begin{document}

\maketitle

\begin{abstract}
This article studies problems of optimal transport, by embedding them in a general functional analytic framework of convex optimization. This provides a unified treatment of a large class of related problems in probability theory and allows for generalizations of the classical problem formulations. General results on convex duality yield dual problems and optimality conditions for these problems. When the objective takes the form of a convex integral functional, we obtain more explicit optimality conditions and establish the existence of solutions for a relaxed formulation of the problem. This covers, in particular, the mass transportation problem and its nonlinear generalizations.
\end{abstract}


\section{Introduction}

Let $S_t$, $t=0,\ldots,T$ be Polish spaces and $S=S_0\times\cdots\times S_T$. Let $M_t$ and $M$ be spaces of $\reals^d$-valued Borel measures on $S_t$ and $S$, respectively, and consider the optimization problem
\begin{equation}\label{d}\tag{D}
\minimize\quad \sum_{t=0}^TG_t^*(\lambda_t) + H^*(\lambda) \quad\ovr\ \lambda\in M,
\end{equation}
where $G^*_t$ and $H^*$ are convex functions on $M_t$ and $M$, respectively and $\lambda_t$ is the marginal of $\lambda$ on $S_t$.

The above covers a wide range of optimization problems encountered in probability theory and finance. In particular, when $T=d=1$, $G^*_t=\delta_{\{\mu_t\}}$ and\footnote{Given a set $C$, its {\em indicator function} $\delta_C$ takes the value $0$ on $C$ and $+\infty$ outside of $C$.} $H^*(\lambda)=\int_S cd\lambda+\delta_{M_+}(\lambda)$ for given $\mu_t\in M_t$ and a lower semicontinuous nonnegative function $c$, we cover the classical Monge--Kantorovich mass transportation problem. Choosing $H^*=\delta_\Lambda$ for a closed convex set $\Lambda\subset M$ of probability measures, we obtain the problem from Strassen~\cite{str65} of finding probability measures with given marginals. When $H^*(\lambda)$ is the entropy relative to a given reference measure, we recover the classical Schr\"odinger problem; see e.g.\ \cite{fg97,leo14} and the references therein. Problems where the effective domain of $H^*$ is contained in the set of martingale measures have been recently proposed in mathematical finance e.g.\ in \cite{bhp13}.

Allowing for more general choices of $G^*_t$ is relevant e.g.\ in economic applications where $\lambda_t$ is not necessarily fixed but can react to demand with an increasing marginal costs of production. In the case of finite $S$, such problems have been extensively studied in \cite{roc84}. In the financial context of \cite{bhp13}, more general convex functionals $G_t^*$ arise naturally when price quotes for derivatives come with bid-ask spreads and finite quantities. 

This paper develops a duality theory for \eqref{d} by embedding it in the general conjugate duality framework of Rockafellar~\cite{roc74}. This provides a unified treatment of a wide range of problems in deriving optimality conditions and criteria for the existence of optimal solutions. The duality approach yields simplified proofs and generalizations of many classical results in applied probability.

As examples, we extend some well-known results on the existence of probability measure with given marginals, on the Schr\"odinger problem and on model-free superhedging of financial derivatives. Our main theorem on problem \eqref{d} yields extensions of the main results of \cite{str65}, \cite{fg97} and \cite{bhp13} to models with general marginal functionals $G_t^*$. 

When the functions $G^*_t$ and $H^*$ have the additional structure of integral functionals, the optimality conditions allow for pointwise characterizations and the problem dual to \eqref{d} allows for a relaxation where the optimum is attained under fairly general conditions. Our existence results extend the existing results on the dual of the Monge--Kantorovich problem to a wider class of problems. In particular, we obtain a necessary and sufficient conditions for optimal transportation plans in mass transportation with capacity constraints. We obtain a similar result for the Schr\"odinger problem which also seems new. 

This paper combines techniques from convex analysis, measure theory and the theory of integral functionals of continuous functions. The general duality results are derived from the functional analytic framework of \cite{roc74} while the theory of integral functionals allows for a more explicit form of optimality conditions and for a relaxation of the problem dual to \eqref{d}. The generality of our setting requires an extended conjugacy theorem for integral functionals proved in the appendix. The attainment of the minimum in the dual of \eqref{d} is established by borrowing techniques from convex stochastic optimization \cite{pp12}.

\section{Conjugate duality}\label{sec:cd}

This section derives \eqref{d} as a dual problem of a convex optimization problem on a Banach space of continuous functions. In some applications it is convenient to allow for unbounded continuous functions so we will follow \cite{str65} and allow for continuous functions that become bounded when scaled by a possibly unbounded continuous function.

Given a continuous $\psi_t:S_t\to[1,\infty)$,
\[
C_t:=\{x_t\in C(S_t;\reals^d)\mid x_t/\psi_t\in C_b(S_t;\reals^d)\}
\]
is a Banach space under the norm $\|x_t\|_{C_t}:=\|x_t/\psi_t\|_{C_b(S_t;\reals^d)}$, where $C_b(S_t;\reals^d)$ is the space of bounded continuous functions with the supremum norm. The space $M_t$ of $\reals^d$-valued finite Borel measures under which $\psi_t$ is integrable may be identified with a linear subspace of the norm dual $C_t^*$ of $C_t$. Indeed, for every $\lambda_t\in M_t$,
\[
x_t\mapsto \int_{S_t} x_td\lambda_t:=\sum_{i=1}^d\int_{S_t} x^i_td\lambda^i_t=\sum_{i=1}^d\int_{S_t} x^i_t/\psi_td(\psi_t\lambda^i_t)
\]
is a continuous linear functional on $C_t$. If $S_t$ is compact, then Riesz representation (see e.g.\ \cite[Theorem~7.10.4]{bog7}) implies that $C_t^*=M_t$ but, in general, the inclusion $M_t\subseteq C_t^*$ may be strict. Similarly, defining
\[
\psi(s):=\sum_{t=0}^T\psi_t(s_t),
\]
the space $M$ of finite $\reals^d$-valued Borel measures on $S$ under which $\psi$ is integrable is a linear subspace of the Banach dual of
\[
C:=\{u\in C(S;\reals^d)\mid u/\psi\in C_b(S;\reals^d)\}.
\]
When $\psi_t$ are bounded, we have $C_t=C_b(S_t;\reals^d)$ and $C=C_b(S;\reals^d)$ the duals of which contain all finite $\reals^d$-valued Borel measures on $S_t$ and $S$, respectively.

%\begin{remark}
%Given a continuous function $\psi:S\mapsto[1,\infty)$,
%\[
%C_{\psi}(S):=\{x\in C(S)\mid x/\psi\in C_b(S)\}
%\]
%is a Banach space under the norm $\|x\|_{C_\psi}=\|x/\psi\|_{C_b}$. The linear mapping $A:C_\psi\to C_b$ defined by $Ax:=x/\psi$ is a continuous bijection with the continuous inverse $A^{-1}x=x\psi$. The topological dual of $C_\psi$ can thus be identified with that of $C_b$ in the sense that for every $\lambda\in C_\psi^*$ there is a unique $\lambda_b\in C_b^*$ such that $\langle x,\lambda\rangle_{C_\psi}=\langle Ax,\lambda_b\rangle_{C_b}$. In other words, the adjoint $A^*:C_b^*\to C_\psi^*$ of $A$ is also a continuous bijection. A Borel measure $\lambda$ is in $C_\psi^*$ if and only if $\psi$ is $\lambda$-integrable.
%\end{remark}

Let $G_t$ be a proper convex function on $C_t$, $t=0,\ldots,T$, let $H$ be a proper convex function on $C$, and consider the problem
\begin{equation*}\label{p}\tag{P}
\begin{alignedat}{2}
&\minimize\quad & & \sum_{t=0}^TG_t(x_t) + H\left(-\sum_{t=0}^Tx_t\comp \pi_t\right)\quad\ovr\ x\in\prod_{t=0}^T C_t,
\end{alignedat}
\end{equation*}
where $x=(x_t)_{t=0}^T$ and $\pi_t(s):=s_t$. The general duality results below depend on the properties of the {\em optimum value function}. 
\[
\varphi(u) := \inf_x\left\{\sum_{t=0}^TG_t(x_t) + H\left(u-\sum_{t=0}^Tx_t\comp \pi_t\right)\right\}
\]
defined on $C$. 

Throughout, we will endow the dual space $C^*$ of $C$ by the weak*-topology. The spaces $C$ and $C^*$ are then in separating duality under the natural bilinear form
\[
\langle u,\lambda\rangle:=\lambda(u).
\]
Similarly for $C_t^*$. It turns out that the conjugate 
\[
\varphi^*(\lambda):=\sup_{u\in C}\{\langle u,\lambda\rangle-\varphi(u)\}
\]
of $\varphi$ can be expressed as
\[
\varphi^*(\lambda)=\sum_{t=0}^T G_t^*(\lambda_t) + H^*(\lambda),
\]
where $G_t^*$ is the conjugate of $G_t$, $H^*$ is the conjugate of $H$ and $\lambda_t\in C_t^*$ denotes the continuous linear functional $x_t\mapsto\langle x_t\comp \pi_t,\lambda\rangle$ on $C_t$, the $t$-th {\em marginal} of $\lambda$. 

The infimum of $\varphi^*$ over $C^*$ equals $-\varphi^{**}(0)$ so if $\varphi$ is lower semicontinuous and the optimum value $\inf\eqref{p}$ of \eqref{p} is finite, then the biconjugate theorem implies that $-\inf\eqref{p}$ equals the optimum value of
\begin{equation}\label{dr}\tag{DR}
\minimize\quad \sum_{t=0}^T G_t^*(\lambda_t) + H^*(\lambda)\quad\ovr\ \lambda\in C^*.
\end{equation}
This may be viewed as a ``relaxation'' of \eqref{d} from the space $M$ of Borel measures to all of $C^*$. Clearly, if $\dom\varphi^*\subseteq M$, then \eqref{dr} coincides with \eqref{d}. The following lemma gives a sufficient condition for this. It is a simple extension of \cite[Lemma~4.10]{leo6} that was formulated for $T=1$ and $\psi_t\equiv 1$.

\begin{lemma}\label{lem:radon}
If $\dom\varphi^*\subset C^*_+$ and $\dom G_t^*\subseteq M_t$ for each $t=0,\ldots,T$, then $\dom\varphi^*\subseteq M$.
\end{lemma}

\begin{proof}
By \cite[Theorem 7.10.6]{bog7}, $\lambda\in \dom\varphi^*$ is a Radon measure (since $S$ is Polish, this is equivalent to being a Borel measure \cite[Theorem 7.1.7]{bog7}) if and only if, for every $\epsilon>0$, there exists a compact $K\subset S$ such that if $u\in C_b$ is zero on $K$, then $|\langle u,\lambda\rangle| \le \epsilon \|u\|$.

Let $\epsilon>0$. By assumption, $\lambda_t\in M_t$ and they are nonnegative since $\dom\varphi^*\subset C^*_+$. By \cite[Theorem~7.1.7]{bog7}, there exist compact sets $K_t$ such that $\lambda_t(K_t^C)<\epsilon/(T+1)$. Let $u \in C_b$ be zero on $\prod K_t$. Since $\lambda$ is an additive set function, and $|u|\le 1_{(\prod K_t)^C}\|u\|_{C_b}$, 
\begin{align*}
|\langle u, \lambda \rangle | &\le \int 1_{(\prod K_t)^C}\|u\| d\lambda\\
&= \lambda\left(\bigcup \pi_t^{-1}(K_t)^C\right)\|u\|\\
&\le  \sum_t \lambda \left(\pi_t^{-1}(K_t)^C\right)\|u\|\\
&\le \sum_t \lambda_t\left(K_t^C\right)\|u\|\\
&= \epsilon\|u\|
\end{align*}
which completes the proof.
\end{proof}

The set of relaxed dual solutions coincides with the {\em subdifferential} $\partial\varphi(0)$ of $\varphi$ at the origin. If $\partial\varphi(0)$ is nonempty, then $\varphi$ is closed at the origin and there is no duality gap. The following result gives a sufficient condition for the existence in \eqref{dr}. It involves the domain
\[
\dom\varphi = \dom H + \{\sum_{t=0}^Tx_t\comp\pi_t\,|\,x_t\in\dom G_t\}
\]
of the optimum value function of \eqref{p} 

\begin{theorem}\label{thm:dual}
If $G_t$ and $H$ be proper lsc functions such that the set
\[
\bigcup_{\alpha>0}\alpha\dom\varphi
\]
% \[
% \bigcup_{\alpha>0}\alpha\left(\dom\varphi -u \right)
% \]
is a nonempty closed linear subspace of $C$, then the optimum in \eqref{dr} is attained, there is no duality gap and an $x$ solves \eqref{p} if and only if there is a $\lambda\in C^*$ such that
\begin{align*}
  \partial G_t(x_t) &\ni \lambda_t,\quad t=0,\ldots,T,\\
  \partial H(-\sum_{t=0}^Tx_t\comp\pi_t) &\ni \lambda,
\end{align*}
and then $\lambda$ solves \eqref{dr}.
\end{theorem}

\begin{proof}
Problem \eqref{p} fits the conjugate duality framework of \cite{roc74} with $X=\prod_{t=0}^T C_t$, $U=C$ and
\[
F(x,u) := \sum_{t=0}^TG_t(x_t) + H\left(u-\sum_{t=0}^Tx_t\comp \pi_t\right).
\]
The associated {\em Lagrangian} $L$ is the convex-concave function defined for each $x\in\prod_{t=0}^TC_t$ and $\lambda\in M$ by %For any $x\in C$ and $\lambda\in C^*$, the Lagrangian can be expressed as
\begin{align*}
  L(x,\lambda) &:= \inf_u\{F(x,u) - \langle u,\lambda\rangle\}\\
&= \sum_{t=0}^TG_t(x_t) - \sum_{t=0}^T\langle x_t\comp \pi_t,\lambda\rangle - H^*(\lambda)\\
&= \sum_{t=0}^TG_t(x_t) - \sum_{t=0}^T\langle x_t,\lambda_t\rangle - H^*(\lambda).
\end{align*}
The conjugate of $F$ can thus be expressed for each $\theta\in\prod_tM_t$ and $\lambda\in M$ as
\begin{align*}
F^*(\theta,\lambda) &= \sup_x\{\langle x,\theta\rangle - L(x,\lambda)\}\\
&= \sup_x\{\sum_{t=0}^T\langle x_t,\theta_t\rangle - \sum_{t=0}^TG_t(x_t) + \sum_{t=0}^T\langle x_t,\lambda_t\rangle + H^*(\lambda)\}\\
&= \sum_{t=0}^T G_t^*(\lambda_t+\theta_t) + H^*(\lambda).
\end{align*}
Thus 
\[
\varphi^* (\lambda) = F^*(0,\lambda) = \sum_{t=0}^T G_t^*(\lambda_t) + H^*(\lambda).
\]
By \cite[Theorem~2.7.1(vii)]{zal2}, $\varphi$ is continuous at the origin relative to $\aff\dom\phi$, so $\partial\varphi(0)\ne\emptyset$ by \cite[Theorem~2.4.12]{zal2}. The claims now follow from Theorems~15 and 16 of \cite{roc74}.
\end{proof}

\begin{remark}\label{rem:cq}
The second condition in Theorem~\ref{thm:dual} holds, in particular, if $0\in\inte\dom\varphi$, which holds, in particular, if
\[
0\in \inte\dom H+\{\sum_{t=0}^Tx_t\comp\pi_t\,|\,x_t\in\dom G_t\}.
\]
In the scalar case $d=1$ this last condition holds, in particular, if $H$ is nondecreasing with $H(0)<\infty$ and there exist $x_t\in\dom G_t$ such that $\sum_{t=0}^Tx_t\comp\pi_t\ge\epsilon\psi$ for some $\epsilon>0$. This is satisfied e.g.\ in the applications of Section~\ref{sec:mt} below where $\dom G_t=C_t$ for all $t$.
\end{remark}

The general results in conjugate duality would also give sufficient conditions for the existence of primal solutions but in many applications, the primal optimum is not attained in $\prod_{t=0}^TC_t$. In Sections~\ref{sec:rel} and \ref{sec:exist} below, we will extend the domain of definition of the primal objective and give sufficient conditions for the attainment of the primal optimum in a larger space of measurable functions.

\section{Examples}\label{sec:ex}

This section illustrates the general results of Section~\ref{sec:cd} by extending three well-known results in measure theory and mathematical finance. From now on, we will use the simplified notation 
\[
\sum_{t=0}^Tx_t :=\sum_{t=0}^Tx_t\comp\pi_t.
\]

\subsection{Probability measures with given marginals}

The first application deals with the classical problem on the existence of probability measures with given marginals. The following extends the existence result of \cite{str65} by allowing for more general conditions on the marginals. As usual, the {\em support function} of a set $D$ in a locally convex space $X$ is the lower semicontinuous convex function $\sigma_D$ on the dual space $V$ of $X$ given by
\[
\sigma_D(v):=\sup_{x\in D}\langle x,v\rangle.
\]

\begin{theorem}\label{thm:str}
Let $\Lambda\subset M$ and $\Lambda_t\subset M_t$ be weakly compact and convex. There exists $\lambda\in \Lambda$ with $\lambda_t\in\Lambda_t$ if and only if
\begin{equation*}
\begin{alignedat}{2}
 \sum_{t=0}^T\sigma_{\Lambda_t}(x_t) + \sigma_{\Lambda}\left(-\sum_{t=0}^Tx_t\right)\ge 0\qquad\forall x\in \prod_{t=0}^T C_t.
\end{alignedat}
\end{equation*}
\end{theorem}

\begin{proof}
This fits Theorem~\ref{thm:dual} with $H=\sigma_{\Lambda}$ and $G_t=\sigma_{\Lambda_t}$. Indeed, by the biconjugate theorem (see e.g.~\cite[Theorem~5]{roc74}), we then have $H^*=\delta_{\Lambda}$ and $G_t^*=\delta_{\Lambda_t}$, so the objective of \eqref{d} is simply the indicator of the set
\[
\{\lambda\in\Lambda\,|\,\lambda_t\in\Lambda_t\}.
\]
The existence is thus equivalent to the optimum value of \eqref{d} being equal to zero. Since $\Lambda$ is bounded, $\dom\varphi=C$, so the domain condition of the Theorem~\ref{thm:dual} is satisfied. Thus, there is no duality gap so $\inf\eqref{d}=0$ if and only if $\inf\eqref{p}=0$, which holds exactly when the condition in the statement holds.
\end{proof}

When $T=d=1$, $\Lambda$ is a subset of probability measures and $\Lambda_t=\{\mu_t\}$ for given probability measures $\mu_t$ on $S_t$, Theorem~\ref{thm:str} reduces to Theorem~7 of \cite{str65}.
%\end{example}

\subsection{Schr\"odinger problem}\label{ssec:sch0}

Let $d=1$ and let $R\in M$ and $\mu_t\in M_t$ be probability measures. The associated {\em Schr\"odinger problem} is the convex minimization problem
\begin{equation*}\label{sp}
\begin{alignedat}{2}
&\minimize\quad & & \int_S\ln(d\lambda/dR)d\lambda \quad\ovr\ \lambda\in M_+(S)\\
&\st\quad & &\lambda\ll R,\quad\lambda_t=\mu_t\quad t=0,\ldots,T.
\end{alignedat}
\end{equation*}
Such problems have been extensively studied in the literature; see e.g.~\cite{csi75} and the references there. 

This fits the format of \eqref{p} with $G_t(x_t)=\int_{S_t}x_td\mu_t$ and
\[
H(u) = \ln\int_Se^udR.
\]
Indeed, $H$ is proper convex lsc function with the conjugate
\[
H^*(\lambda) =
\begin{cases}
\int_S\ln(d\lambda/dR)d\lambda & \text{if $\lambda\in\P$,}\\
+\infty & \text{otherwise},
\end{cases}
\]
where $\P\subset M$ is the set of probability measures. The expression of the conjugate is derived e.g.\ in \cite[Section~3]{roc71} under the assumption that $S$ is a compact Hausdorff space and $\psi=1$. Combined with with Theorem~\ref{thm:if} below, the same argument works in the case of Polish $S$ and general $\psi$.

Allowing for general proper lsc convex $G_t$, gives rise to the following generalized formulation of the Schr\"odinger problem
\begin{equation}\label{spg}
\begin{alignedat}{2}
&\minimize\quad & & \sum_{t=0}^TG_t^*(\lambda_t) + \int_S\ln(d\lambda/dR)d\lambda \quad\ovr\ \lambda\in M_+(S)\\
&\st\quad & &\lambda\ll R.
\end{alignedat}
\end{equation}
This allows for situations where the marginals are not known exactly. Theorem~\ref{thm:dual} combined with Remark~\ref{rem:cq} gives the following.

\begin{theorem}\label{thm:spd}
Assume that there exist $x_t\in\dom G_t$ such that $\sum_{t=0}^Tx_t\ge\epsilon\psi$ for some $\epsilon>0$. Then the optimum in \eqref{spg} is attained and the optimum value coincides with the negative of the optimum value of
\[
\minimize\quad \sum_{t=0}^TG_t(x_t) + \ln\int_S\exp\left(-\sum_{t=0}^Tx_t\right)dR\quad\ovr\ x\in\prod_{t=0}^TC_t.
\]
\end{theorem}

When $T=1$ and $G_t(x_t)=\int_{S_t}x_td\mu_t$, we recover the dual of the Schr\"odinger problem studied in \cite{leo14}. In Section~\ref{ssec:sch} below, we will associate \eqref{sp} with another dual problem for which the optimum is attained. This yields necessary and sufficient conditions for the minimizers of the Schr\"odinger problem. This provides a duality proof of the optimality conditions given in \cite[Theorem~3.43]{fg97}.

\subsection{Model-independent superhedging}

Let $d=1$, $S_t=\reals^n$ and $\psi_t(s_t)=1+|s_t|$ for all $t$ and $H=\delta_{\C_{\hat u}}$, where
\[
\C_{\hat u} := \{u\in C\,|\,\exists z\in\N:\ \hat u(s)+u(s)\le \sum_{t=0}^{T-1}z_t(s^t)\cdot\Delta s_{t+1}\}
\]for an upper semicontinuous function $\hat u$ and
\[
\N:=\{(z_t)_{t=0}^{T-1}\,|\,z_t\in\L^\infty(S^t;\reals^n)\quad t=0,\ldots,T\},
\]
where $S^t:=S_0\times\cdots\times S_t$. Problem \eqref{p} becomes 
\begin{equation}\label{sssh}
\begin{alignedat}{2}
&\minimize\quad & & \sum_{t=0}^TG_t(x_t) \quad\text{over}\quad x\in\prod_{t=0}^TC_t,\ z\in\N\\
&\st & & \quad \hat u(s) \le \sum_{t=0}^{T-1}z_t(s^t)\cdot\Delta s_{t+1} + \sum_{t=0}^Tx_t(s_t)\quad\forall s.
\end{alignedat}
\end{equation}
This can be interpreted as a problem of optimal superhedging $\hat u$ in a financial market where $G_t$ gives the cost of buying an $s_t$-dependent cash-flow $x_t$ paid out at time $t$ and the sum involving $z$ represents the gains from a self-financing trading strategy described by $z$. When
\[
G_t(x_t)=\int_{S_t}x_td\mu_t
\]
for given probability measures $\mu_t$, we recover the superhedging problem studied in \cite{bhp13}. Nonlinear functions $G_t$ arise naturally in practice where one faces bid-ask spreads and price quotes are available only for finite quantities.

%For every $z\in\N$, the function $s\mapsto \sum_{t=0}^{T-1}z_t(s^t)\cdot\Delta s_{t+1}$ belongs to $C$, so $\C_{\hat u}=\Gamma+\C_0$, where

We will denote the set of nonnegative {\em martingale measures} by
\[
\M:=\{\lambda\in M_+\,|\, \int_S\sum_{t=0}^Tz_t(s^t)\cdot\Delta s_{t+1}d\lambda = 0\quad\forall z\in\N\}.
\]

\begin{lemma}
Assume that $\hat u\le K\psi$ for some $K\in\reals$. Then for $\lambda\in M$, the conjugate of $H$ can be expressed as
\[
H^*(\lambda) =
\begin{cases}
  -\int_S\hat ud\lambda & \text{if $\lambda\in\M$},\\
  +\infty & \text{otherwise}.
\end{cases}
\]
\end{lemma}

\begin{proof}
It is clear that $H^*(\lambda)=+\infty$ unless $\lambda\ge 0$. For $\lambda\ge 0$,
\[
H^*(\lambda)\le\sup_{z\in\N}\int_S\sum_{t=0}^Tz_t(s^t)\cdot\Delta s_{t+1}d\lambda -\int_S\hat ud\lambda =
\begin{cases}
  -\int_S\hat ud\lambda & \text{if $\lambda\in\M$},\\
  +\infty & \text{otherwise}.
\end{cases}
\]
On the other hand,
\[
H^*\ge\sigma_\Gamma+\sigma_{\C^c_0},
\]
where $\Gamma=\{u\in C\,|\,u\le-\hat u\}$ and
\[
\C_0^c:=\{u\in C\,|\,\exists z\in\tilde\N:\ u(s)\le\sum_{t=0}^{T-1}z_t(s^t)\cdot\Delta s_{t+1}\}
\]
with $\tilde\N\subset\N$ denoting the continuous bounded strategies. When $\hat u\le K\psi$ for some $K\in\reals$, then for $\lambda\ge 0$,
\[
\sigma_\Gamma(\lambda) = -\int_S\hat ud\lambda,
\]
by Theorem~\ref{thm:if} below. By standard approximation arguments, $\sigma_{\C_0^c}=\delta_\M$ (see e.g.~\cite[page~435]{str65} or \cite[Lemma~2.3]{bhp13}).
\end{proof}

When $\dom G_t^*\subset M_t$ for all $t=0,\ldots,T$, the feasible dual solutions are in $M$, by Lemma~\ref{lem:radon}, so problem \eqref{d} can be written as
\begin{equation}\label{dsssh}
\minimize\quad\sum_{t=0}^TG_t^*(\lambda_t) - \int_S\hat ud\lambda\quad\ovr\quad\lambda\in \M.
\end{equation}
Combining Theorem~\ref{thm:dual} with Remark~\ref{rem:cq} gives the following.

\begin{theorem}\label{thm:sh}
Assume that $\dom G_t^*\subset M_t$ for all $t=0,\ldots,T$, that $\hat u\le K\psi$ for some $K\in\reals$ and that \eqref{sssh} remains feasible when $\hat u$ is increased by $\epsilon\psi$ for some $\epsilon>0$. Then the optimum in \eqref{dsssh} is attained and the optimum value coincides with the negative of the optimum value of \eqref{sssh}.
\end{theorem} 

When $G_t(x_t)=\int x_td\mu_t$ for given $\mu_t\in M_t$, the feasibility condition is trivially satisfied and we recover Theorem~1.1 of \cite{bhp13}. In fact, Theorem~\ref{thm:sh} is slightly sharper than \cite[Theorem~1.1]{bhp13} since we obtain the absence of a duality gap for continuous functions $x_t$.

%\begin{remark}[Calibration to price quotes]\label{rem:quotes}
%In the financial context, one typically has
%\[
%G_t(x_t) = \inf_{\alpha\in\reals^{J_t}}\{\sum_{j\in J_t}S^j_t(\alpha^j)\,|\, x_t=\sum_{j\in J_t}\alpha^jx_t^j\},
%\]
%where $x_t^j$ is the payoff of one unit of a financial contract, whose market quotes are such the total cost of buying $\alpha$ units of the contract is $S_t^j(\alpha)$. We get
%\[
%G_t^*(\theta_t) = \sum_{j\in J_t}(S_t^j)^*\left(\langle x^j,\theta_t\rangle\right).
%\]
%If each $S_t^j$ is the support function of an interval $[\underbar s_t^j,\bar s_t^j]$ (infinite supply at the best bid and best ask), $G_t^*$ is the indicator function of the set
%\[
%\{\theta_t\,|\, \langle x^j,\theta_t\rangle\in[\underbar s_t^j,\bar s_t^j]\quad \forall j\in J_t\}.
%\]
%\end{remark}

% \begin{remark}[Stochastic dominance]
% Assuming that each $G_t$ is nondecreasing, it is optimal to choose $x$ and $z\in\N$ such that, for all $t$, each function
% \begin{align*}
% w_t(s):= x_t(s_t)-\hat u(s)\quad t=T,\\
% w_{t}(s^{t}):=w_{t+1}(s^{t},s_{t})+x_{t}(s_t) \quad t=T-1,\dots 0,
% \end{align*}
% is convex w.r.t. $s_t$ and  $-z_{t-1}(s^t)\in\partial_{s_t} w_t(s^t)$, and ??% $x_0(s_0) := -w_0(s_0)$,
% \end{remark}

We denote by $C^c$ the subset of convex functions in $C$. Allowing for unbounded continuous functions is essential here as the only bounded convex functions are the constant functions. The following corollary of Theorem~\ref{thm:sh} extends \cite[Theorem~8]{str65} on the existence of martingale measures with given marginals.

\begin{corollary}\label{cor:exm}
Let $\Lambda_t\subset M_t$ be weakly closed convex sets of probability measures.  There exists $\lambda\in \M$ with $\lambda_t\in\Lambda_t$ if and only if
\begin{equation*}
\begin{alignedat}{2}
\sum_{t=0}^T \sigma_{\Lambda_t}(w_t-w_{t+1}) \ge 0
\end{alignedat}
\end{equation*}
for all  $w\in\prod_{t=0}^{T+1}C_t^c$ with $w_0\ge 0$ and $w_{T+1}=0$.
\end{corollary}

\begin{proof}
Let $G_t=\sigma_{\Lambda_t}$ and $\hat u=0$ in Theorem~\ref{thm:sh}. Given an $x\in\prod_{t=0}^TC_t$, define $w_r\in C_r$ for $r=0,\ldots,T$ by
\[
w_r(s_r):=\sum_{t=r}^Tx_t(s_r).
\]
If $w_0\ge 0$ and $w_r$ is convex for each $r$, then $x$ is feasible. Indeed, if for some $r$,
\begin{equation}\label{constr}\tag{$H_r$}
0 \le \sum_{t=0}^{r-1}z_t(s^t)\cdot\Delta s_{t+1} + \sum_{t=0}^{r-1}x_t(s_t) + w_r(s_r)\quad\forall s\in S
\end{equation}
and we choose $-z_r(s^r)\in\partial w_{r+1}(s_r)$, then
\[
0\le z_r(s_r)\cdot\Delta s_{r+1} + w_{r+1}(s_{r+1})-w_{r+1}(s_r),
\]
which combined with \eqref{constr} gives $(H_{r+1})$. For $r=0$, \eqref{constr} simply means $w_0\ge 0$.

On the other hand, since $\sigma_{\Lambda_t}$ are nondecreasing, it is optimal to choose $x_t$ so that $w_r$ are convex. Indeed, for $r=T$ this is clear as \eqref{constr} implies that the optimal $x_T$ is given as a pointwise supremum of affine functions of $s_T$ whose gradients are in $\L^\infty$. If $w_t$ is convex for $t>r$, then \eqref{constr} is necessary and sufficient for feasibility so it is optimal to choose $w_r$ as small as possible subject to \eqref{constr}, which again means that $w_r$ is convex. Moreover, since $z_t$ are bounded, $w_r\in C_r^c$. The optimum value thus equals that of
\begin{equation*}
\begin{alignedat}{2}
&\minimize\quad & & \sum_{t=0}^TG_t(w_t-w_{t+1}) \quad\ovr\quad w\in\prod_{t=0}^TC_t^c,\\
  & \st & & \quad w_0\ge 0,
\end{alignedat}
\end{equation*}
where $w_{T+1}:=0$.
\end{proof}
Note that if $\Lambda_t=\{\mu_t\}$ for each $t$, then 
\[
\sum_{t=0}^T \sigma_{\Lambda_t}(w_t-w_{t+1})=\sum_{t=0}^T\int_{\reals^n} w_td(\mu_t-\mu_{t-1}),
\]
and Corollary~\ref{cor:exm} reduces to \cite[Theorem~8]{str65}, which says that there exists a martingale measure with marginals $\mu_t$ if and only if $\mu_t$ are in convex order.

\section{Integral functionals}\label{sec:if}

From now on, we assume extra structure on $G_t$ and $H$ that will 
\begin{enumerate}
\item
allow us to write the optimality conditions in a more explicit pointwise form,
\item
suggests a natural relaxation of problem \eqref{p} to a larger space of measurable functions where the infimum is more likely to be attained.
\end{enumerate}
More precisely, we assume that each $G_t$ is an integral functional of the form
\[
G_t(x_t) = \int_{S_t}g_t(x_t(s_t),s_t)d\mu_t(s_t)+\delta_{C(D_t)}(x_t),
\]
where $\mu_t$ is a probability measure on $S_t$, $g_t$ is  {\em a convex $\B(S_t)$-normal integrand\footnote{This means that the set-valued mapping $s_t\mapsto\{(x_t,\alpha)\in\reals^d\times\reals\,|\,g_t(x_t,s_t)\le\alpha\}$ is $\B(S_t)$-measurable and closed convex-valued; see e.g.\ \cite[Chapter~14]{rw98}.} on $\reals^d$} , $D_t(s_t):=\cl\dom g_t(\cdot,s_t)$ and 
\[
C(D_t):=\{u\in C_t\,|\, u(s_t)\in D_t(s_t)\ \forall s_t\in S_t\}
\]
is the set of selections of $D_t$. Similarly, we assume that
\[
H(u) = \int_Sh(u(s),s)d\mu(s)+\delta_{C(D)}(u)
\]
where $\mu$ is a probability measure on $S$, $h$ is a convex $\B(S)$-normal integrand on $\reals^d$ and $D(s)=:\cl\dom h(\cdot,s)$. 

We define a function $h^\infty$ on $\reals^d\times S$ by setting $h^\infty(\cdot,s)$ equal to the recession function of $h(\cdot,s)$. Recall that the recession function of a lsc convex function $k$ is given by
\[
k^\infty(x) := \sup_{\alpha>0}\frac{k(\bar x+\alpha x) - k(\bar x)}{\alpha},
\]
which is independent of the choice $\bar x\in\dom k$; see \cite[Theorem~8.5]{roc70a}. By \cite[Exercise~14.54(a)]{rw98}, $h^\infty$ is a convex $\B(S)$-normal integrand on $\reals^d$. Recall that a set-valued mapping $D:S\tos\reals^d$ is {\em inner semicontinuous} (isc) if the inverse image under $D$ of every open set in $O\subset\reals^d$ is open in $S$, the inverse image being defined by
\[
D^{-1}(O):=\{s\in S\,|\,D(s)\cap O\ne\emptyset\}.
\]
The following result characterizes the conjugate and the subdifferential of $H$. Its proof can be found in the appendix. 

Given a $\lambda\in M$, we denote its absolutely continuous and singular parts, respectively, with respect to $\mu$ by $\lambda^a$ and $\lambda^s$. The {\em normal cone} of $D(s)$ at a point $u$ is defined as the subdifferential of $\delta_{D(s)}$ at $u$. More explicitly, it is the closed convex cone $N_{D(s)}(u)$ given by 
\[
N_{D(s)}(u)=
\{y\in\reals^d\,|\,(u'-u)\cdot y\le 0\quad\forall u'\in D(s)\}
\]
for $u\in D(s)$ and $N_{D(s)}(u)=\emptyset$ for $u\notin D(s)$.

\begin{theorem}\label{thm:if}
Assume that $D(s):=\dom h(\cdot,s)$ is isc, $\cl\dom H=C(D)$ and that $H$ is finite and continuous at some $u\in C$. Then $H$ is a proper convex lsc function and the restriction to $M$ of its conjugate is given by
\[
H^*(\lambda) = \int_S h^*(d\lambda^a/d\mu)d\mu + \int_S(h^*)^\infty(d\lambda^s/d|\lambda^s|)d|\lambda^s|.
\]
Moreover, $\lambda\in\partial H(u)\cap M$ if and only if
\begin{align*}
d\lambda^a/d\mu &\in\partial h(u)\quad\mu\text{-a.e.}\\
d\lambda^s/d|\lambda^s| &\in N_D(u)\quad|\lambda^s|\text{-a.e.}
\end{align*}
If $\dom H =C$, then $\dom H^*\subseteq \{\lambda \in M\mid \lambda\ll \mu\}$.
\end{theorem}

%Given a Borel measure $\lambda_t$ on $S_t$, we will denote the absolutely continuous and singular parts, respectively, of $\lambda_t$ with respect to $\mu_t$ by $(\lambda_t)^a$ and $(\lambda_t)^s$. 
Combining Theorem~\ref{thm:if} with Lemma~\ref{lem:radon} gives the following.

\begin{corollary}\label{cor:if}
Assume that $H$ and $G_t$ all satisfy the assumptions of Theorem~\ref{thm:if} and that $\dom H=C$ or $\dom G_t=C_t$ for all $t=0,\ldots,T$. Then $\dom\varphi^*\subset M$, $\lambda_t\ll\mu_t$ for all $\lambda\in\dom\varphi^*$ and the optimality conditions in Theorem~\ref{thm:dual} can be written as
\begin{align*}
d\lambda_t/d\mu_t &\in\partial g_t(x_t)\quad\mu_t\text{-a.e.}\quad t=0,\dots,T\\
%d(\lambda_t)^s/d|(\lambda_t)^s| &\in N_{D_t}(x_t)\quad|(\lambda_t)^s|\text{-a.e.}\quad t=0,\dots,T,\\
d\lambda^a/d\mu &\in\partial h(-\sum_{t=0}^T x_t)\quad\mu\text{-a.e.}\\
d\lambda^s/d|\lambda^s| &\in N_D(-\sum_{t=0}^T x_t)\quad|\lambda^s|\text{-a.e.}.
\end{align*}
\end{corollary}

The optimality conditions characterize the optimal primal-dual pairs of solutions but in many applications, the primal optimum is not attained in the space of continuous functions. This motivates a relaxation of the primal problem to a larger space where the optimal solutions are more likely to exist.

\section{Relaxation of the primal problem}\label{sec:rel}

In general, primal solutions do not exist in the space of continuous functions but we will establish the existence of solutions in a larger space of measurable functions when the functionals $G_t$ and $H$ are integral functionals as in Section~\ref{sec:if} above and  $\mu_t$ is the $t$-th marginal of $\mu$.

More precisely, %given a family $\H$ of Borel sets in $S$, 
we study the problem
\begin{equation}\label{pr}\tag{$\text{PR}$}
\minimize\quad \int_S\left[\sum_{t=0}^T g_t(x_t) +h(-\sum_{t=0}^T x_t)\right]d\mu \quad\ovr\quad x\in\Phi,
  %  & \st & & \quad \ge 0,
\end{equation}
where
\begin{align*}
\Phi := \{x\in\prod_{t=0}^T\L^0_t\,|\, x_t\in D_t,\ -\sum_{t=0}^T x_t\in D\quad (\mu_t)_{t=0}^T\text{-a.e.}\},
%\Phi_\lambda = \{x\in\prod_{t=0}^T\L^0_t\,|\, x_t\in D_t,\ -\sum_{t=0}^T x_t\in D\ |\lambda|\text{-a.s.}\}.
\end{align*}
%where $\Phi_\lambda$ is the set of $x\in\prod_{t=0}^T\L^0_t$ such that
%\begin{align*}
%  \Phi_\lambda := \{x\in\prod_{t=0}^T\L^0_t\,|\, x_t\in\L^0(|\lambda|_t;D_t)\cap\dom G_t,\quad -\sum_{t=0}^T x_t\in \L^0(|\lambda|;D)\cap\dom H\}.
%\end{align*}
where $\L^0_t:=\L^0(S_t,\B(S_t);\reals^d)$ and $(\mu_t)_{t=0}^T$-almost everywhere means that the property holds on a Cartesian product of sets of full measure on $S_t$.

\begin{lemma}\label{lem:mut}
If $A\in\B(S)$ occurs $(\mu_t)$-almost everywhere then $\mu(A)=1$.
\end{lemma}

\begin{proof}
By definition, $A$ occurs $(\mu_t)$-almost everywhere if there exist $A_t\in S_t$ with $\mu_t(A_t)=1$ such that $\prod_{t=0}^T A_t\subset A$. Noting that $\prod_{t=0}^T A_t = \bigcap_{t=0}^T\pi_t^{-1}(A_t)$, where $\pi_t$ is the projection $s\mapsto s_t$, we get
\begin{align*}
  \mu((\prod_{t=0}^T A_t)^c) &= \mu\left(\bigcup_{t=0}^T(\pi_t^{-1}(A_t))^c\right)\le \sum_{t=0}^T\mu\left(\pi_t^{-1}(A_t^c))\right) = \sum_{t=0}^T\mu_t(A_t^c), 
\end{align*}
where $\mu_t(A_t^c)=0$.
\end{proof}

Sufficient conditions for attainment of the minimum in \eqref{pr} will be given in Theorem~\ref{thm:exist} below. Clearly, the optimum value of \eqref{pr} minorizes that of \eqref{p}. To guarantee that the optimum value of \eqref{pr} is still greater than $-\inf\eqref{d}$ we will assume the following.

\begin{assumption}\label{ass1}
Feasible $x$ in \eqref{pr} and $\lambda$ in \eqref{d} satisfy
\[
\int_S\left[\sum_{t=0}^T g_t(x_t)\right]d\mu<\infty\quad\text{and}\quad \int_Sh(-\sum_{t=0}^T x_t)d\mu<\infty
\]
and
\[
\int_S\left[\sum_{t=0}^Tx_t\cdot\frac{d\lambda_t}{d\mu_t}\right]d\mu = \int_S\left[\sum_{t=0}^T x_t\cdot\frac{d\lambda}{d|\lambda|}\right]d|\lambda|.
\]
\end{assumption}

Sufficient conditions for Assumption~\ref{ass1} will be given at the end of this section. The following statement shows that \eqref{pr} can indeed be considered as a valid dual to \eqref{d}.

\begin{theorem}\label{thm:oc}
Assume that the normal integrands $g_t$ and $h$ satisfy the conditions of Corollary~\ref{cor:if} and that Assumption~\ref{ass1} holds. Then
\[
\inf\eqref{d}\le\inf\eqref{pr}\le\inf\eqref{p}
\]
and feasible solutions $x$ in \eqref{pr} and $\lambda$ in \eqref{d} are optimal with $\inf\eqref{pr}=-\inf\eqref{d}$ if and only if
\begin{align*}
d\lambda_t/d\mu_t &\in\partial g_t(x_t)\quad\mu_t\text{-a.e.}\quad t=0,\dots,T\\
%d(\lambda_t)^s/d|(\lambda_t)^s| &\in N_{\cl\dom g_t}(x_t)\quad|(\lambda_t)^s|\text{-a.e.}\quad t=0,\dots,T,\\
d\lambda^a/d\mu &\in\partial h(-\sum_{t=0}^T x_t)\quad\mu\text{-a.e.}\\
d\lambda^s/d|\lambda^s| &\in N_D(-\sum_{t=0}^T x_t)\quad|\lambda^s|\text{-a.e.}
\end{align*}
\end{theorem}

\begin{proof}
Let $x$ and $\lambda$ be feasible in \eqref{pr} and \eqref{d}, respectively. By Lemma~\ref{lem:mut}, the condition $\lambda_t\ll\mu_t$ implies
\[
x_t\in D_t,\ -\sum_{t=0}^T x_t\in D\quad \lambda\text{-a.e.}
\]
Thus, by Fenchel's inequality,
\begin{align}
g_t(x_t)+g_t^*(d\lambda_t/d\mu_t) &\ge  x_t\cdot(d\lambda_t/d\mu_t)\quad\mu_t\text{-a.e.}\label{1}\\
h(-\sum_{t=0}^Tx_t)+ h^*(d\lambda^a/d\mu) & \ge (-\sum_{t=0}^T x_t)\cdot (d\lambda^a/d\mu)\quad\mu\text{-a.e.}\label{2}\\
%(g_t^*)^\infty(d(\lambda_t)^s/d|(\lambda_t)^s|) &\ge x_t\cdot(d(\lambda_t)^s/d|(\lambda_t)^s|)\label{3}\\
(h^*)^\infty(d\lambda^s/d|\lambda^s|) & \ge (-\sum_{t=0}^T x_t)\cdot (d\lambda^s/d|\lambda^s|)\quad|\lambda^s|\text{-a.e.}.\label{4}
\end{align} 
Summing up, \eqref{1} gives
\[
\sum_{t=0}^Tg_t(x_t) + \sum_{t=0}^Tg_t^*(d\lambda_t/d\mu_t) \ge \sum_{t=0}^Tx_t\cdot(d\lambda_t/d\mu_t)\quad(\mu_t)\text{-a.e.},
\]
where, by Lemma~\ref{lem:mut}, the inequality holds $\mu$-almost everywhere as well. Integrating, we get
\[
\int_S\left[\sum_{t=0}^Tg_t(x_t)\right]d\mu + \int_S\left[\sum_{t=0}^Tg_t^*(d\lambda_t/d\mu_t)\right]d\mu \ge \int_S\left[\sum_{t=0}^Tx_t\cdot(d\lambda_t/d\mu_t)\right]d\mu
\]
On the other hand, \eqref{2} and \eqref{4} give
\[
\int_S h(-\sum_{t=0}^Tx_t)d\mu + H^*(\lambda) \ge \int_S(-\sum_{t=0}^T x_t)\cdot\frac{d\lambda}{d|\lambda|}d|\lambda|.
\]
By the first part of Assumption~\ref{ass1}, the left hand sides of the above two inequalities are finite so
\begin{multline*}
  \int_S\left[\sum_{t=0}^Tg_t(x_t) + h(-\sum_{t=0}^Tx_t)\right]d\mu + \sum_{t=0}^TG_t^*(\lambda_t) + H^*(\lambda) \\
  \ge \int_S\left[\sum_{t=0}^Tx_t\cdot(d\lambda_t/d\mu_t)\right]d\mu + \int_S(-\sum_{t=0}^T x_t)\cdot\frac{d\lambda}{d|\lambda|}d|\lambda|,
\end{multline*}
where the right hand side vanishes by the second part of Assumption~\ref{ass1}. Thus, $-\inf\eqref{d}\le\inf\eqref{pr}$. The above also shows that this holds as an equality if and only if \eqref{1}--\eqref{4} hold as equalities almost everywhere, which in turn is equivalent to the subdifferential conditions in the statement; see e.g.~\cite[Theorem~23.5]{roc70a}.
\end{proof}

The following lemma gives sufficient conditions for the first part of Assumption~\ref{ass1}.

\begin{lemma}\label{lem:ass1}
Assume that there exist $\bar v\in\L^\infty$, $\beta\in\L^1$ and $\delta>0$ such that $g^*_t(\bar v(s),s)\le\beta(s)$ and $h^*(\bar v(s)+v,s)\le\beta(s)$ for $v\in\reals^d$ with $|v|\le\delta$. Then, the first part of Assumption~\ref{ass1} holds. If in addition, $\mu=\prod_{t=0}^T\mu_t$, then $x_t\in\L^1_t$ for every feasible $x$ in \eqref{pr}.
\end{lemma}

\begin{proof}
By Fenchel's inequality,
\begin{align*}
\sum_{t=0}^Tg_t(x_t,s_t) + h(-\sum_{t=0}^Tx_t,s) &\ge \sum_{t=0}^T[\bar v(s)\cdot x_t-g_t^*(\bar v(s),s)]\\
&\quad - (\bar v(s)+v)\cdot\sum_{t=0}^Tx_t - h^*(\bar v(s)+v,s)\\
&\ge -v\cdot \sum_{t=0}^Tx_t - (T+2)\beta(s).
\end{align*}
Since this holds for any $v\in\reals^d$ with $|v|\le\delta$, the sum $\sum_{t=0}^Tx_t$ is $\mu$-integrable if $x$ is feasible in \eqref{pr}. By Fenchel's inequality again,
\[
\sum_{t=0}^Tg_t(x_t)\ge \bar v(s)\cdot\sum_{t=0}^Tx_t-(1+T)\beta(s)\quad\text{and}\quad h(-\sum_{t=0}^Tx_t) \ge \bar v(s)\cdot\sum_{t=0}^Tx_t-\beta(s)
\]
so the first part of Assumption~\ref{ass1} is satisfied. If $\mu=\prod_{t=0}^T \mu_t$, then by Fubini's theorem, $\mu$-integrability of $\sum_{t=0}^Tx_t$ implies that each $x_t$ is $\mu_t$-integrable.
\end{proof}

%\begin{lemma}\label{lem:perp}
%If $\lambda_t=\mu_t$,
%\[
%\int_S\left[\sum_{t=0}^Tx_t\right]^+d\mu < \infty \quad\text{and}\quad \int_S\left[-\sum_{t=0}^T x_t\right]^+d\lambda < \infty
%\]
%then
%\[
%\int_S\sum_{t=0}^Tx_td\mu = \int_S\sum_{t=0}^Tx_td\lambda.
%\]
%\end{lemma}
%\begin{proof}
%Let $x^\nu_t= \min\{\max\{\nu,x_t\}-\nu\}$.
%\end{proof}

The second part of Assumption~\ref{ass1} clearly holds when feasible solutions $x$ of \eqref{pr} have $x_t$ bounded. More generally, it holds if each $x_t$ is $\lambda_t$-integrable. This holds under all the assumptions of Lemma~\ref{lem:ass1}, when $d=1$ and feasible $\lambda$ satisfy $\lambda_t=\mu_t$. This last condition holds in problems with given marginals; see Section~\ref{sec:mt} below.

In some problems it is essential not to require the integrability of $x_t$; see Section~\ref{ssec:sch} below. The following lemma addresses such situations but, interestingly, the argument only works when $T=d=1$. The idea for the proof is taken from that of \cite[Corollary~3.15]{fg97}.

\begin{lemma}\label{lem:perp}
If $T=d=1$ and feasible $\lambda$ satisfies $\lambda\in M_+$ and $\lambda_t=\mu_t$, then the first part of Assumption~\ref{ass1} implies the second part.
%then the sum is both $\mu$- and $\lambda$-integrable and
%\[
%\int_S\left(\sum_{t=0}^Tx_t\right)d\mu = \int_S\left(\sum_{t=0}^Tx_t\right)d\lambda.
%\]
\end{lemma}

\begin{proof}
The proof of Theorem~\ref{thm:oc} shows that when $T=d=1$, $\mu,\lambda\in M_+$ and $\lambda_t=\mu_t$, feasible solutions $x$ and $\lambda$ satisfy
\[
\int_S\left[\sum_{t=0}^Tx_t\right]^+d\mu < \infty \quad\text{and}\quad \int_S\left[\sum_{t=0}^T x_t\right]^-d\lambda < \infty
\]
Let $x_t^\nu$ be the pointwise projection of $x_t$ to the unit ball of radius $\nu$. We have
\[
\int_S\left(\sum_{t=0}^Tx^\nu_t\right)d\mu=\sum_{t=0}^T\int_{S_t}x^\nu_td\mu_t=\sum_{t=0}^T\int_{S_t}x^\nu_td\lambda_t=\int_S\left(\sum_{t=0}^Tx^\nu_t\right)d\lambda
\]
and when $T=1$,
\[
[\sum_{t=0}^Tx^\nu_t]^+\le[\sum_{t=0}^Tx_t]^+\quad\text{and}\quad[\sum_{t=0}^Tx^\nu_t]^-\le[\sum_{t=0}^Tx_t]^-
\]
so, by Fatou's lemma,
\[
\int_S\left(\sum_{t=0}^Tx_t\right)d\mu \ge\limsup\int_S\left(\sum_{t=0}^Tx^\nu_t\right)d\mu\ge\liminf\int_S\left(\sum_{t=0}^Tx^\nu_t\right)d\lambda \ge\int_S\left(\sum_{t=0}^Tx_t\right)d\lambda.
\]
This implies
\[
\int_S\left[\sum_{t=0}^Tx_t\right]^-d\mu < \infty \quad\text{and}\quad \int_S\left[\sum_{t=0}^T x_t\right]^+d\lambda < \infty,
\]
so the same argument gives the reverse inequality.
\end{proof}

\section{Existence of relaxed primal solutions}\label{sec:exist}

We now turn to the existence of solutions in the relaxed problem \eqref{pr}. We start more abstractly by considering problems of the form
\begin{equation}\label{pf}\tag{$\bar P$}
\minimize\quad \int_S f(x)d\mu \quad\ovr\quad x\in\Phi,
  %  & \st & & \quad \ge 0,
\end{equation}
where $f$ is a convex normal $\B(S)$-integrand on $\reals^{(1+T)d}$ and 
\[
\Phi:= \left\{x\in\prod_{t=0}^T \L^0_t \midb x \in \cl\dom f\quad (\mu_t)_{t=0}^T\text{-a.e.}\right\}.
\] 

Problem \eqref{pr} fits \eqref{pf} with
\begin{equation}\label{fgh}
f(x,s) = \sum_{t=0}^T g_t(x_t,s_t)+h(-\sum_{t=0}^Tx_t,s)
\end{equation}
under the following.

\begin{assumption}\label{ass2}
The set
\[
\{x\in\reals^{(1+T)d}\,|\,x_t\in\rinte D_t(s_t),\ -\sum_{t=0}^Tx_t\in\rinte D(s)\}
\]
is nonempty for every $s\in S$.
\end{assumption}

Indeed, by \cite[Propositions~14.44(d) and 14.45(a)]{rw98}, $f$ defined by \eqref{fgh} is a normal integrand, and, by \cite[Theorem~9.3]{roc70a}, Assumption~\ref{ass2} implies
\[
\cl\dom f(\cdot,s) = \{x\,|\,x_t\in D_t(s),\ -\sum_{t=0}^Tx_t\in D(s)\}.
\]
Assumption~\ref{ass2} is automatically satisfied if $D_t(s_t)=\reals^d$ for all $t$ since $\rinte D(s)\ne\emptyset$, by \cite[Theorem~6.2]{roc70a}.

Except for the filtration property, problem \eqref{pf} is similar to the general stochastic optimization problem studied in \cite{pp12}. The following variant of \cite[Theorem~2]{pp12} gives sufficient conditions for the existence of solutions in \eqref{pf}. Its proof uses \cite[Corollary~8.6.1]{roc70a} which says that if $x\in\reals^{(1+T)d}$ is such that $f^\infty(x,s)\le 0$ and $f^\infty(-x,s)\le 0$, then $f(\bar x+x,s)=f(\bar x,s)$ for every $\bar x\in\dom f(\cdot,s)$. The Borel sigma-algebra generated on $S$ by the projection of $s$ to $s_t$ will be denoted by $\F_t$. 

The statements below involve the set
\[
N:=\{x\in\reals^{(1+T)d}\mid \sum x_t=0\}.
\]

\begin{theorem}\label{thm:exist}
Assume that $\prod_{t=0}^T\mu_t\ll\mu$, there exists $m\in \L^1(S,\F,\mu)$ such that
\[
f(x,s)\ge m(s)\quad\forall x\in\reals^{(1+T)d}\ \text{$\mu$-a.e.},
\]
and that for every $s\in S$
\begin{equation}\label{eq:exist}
\{x\in\reals^{(1+T)d}\mid f^\infty(x,s)\le 0\} = N.
\end{equation}
Then \eqref{pf} has a solution.
\end{theorem}

\begin{proof}
Let $x$ be feasible in \eqref{pf}. Since the set $N:=\{x\in\reals^{(1+T)d}\mid \sum x_t=0\}$ is linear, condition \eqref{eq:exist} implies, by \cite[Corollary~8.6.1]{roc70a}, that $f(\cdot,s)$ is constant in the directions of $N$. Let
\[
N_t:=\{x_t\in\reals^d\mid \exists z\in\reals^{d\prod_{t=0}^T(1+T)}: (0,\dots,0,x_t,z_{t+1},\dots,z_T)\in N\}.
\]
The $s$-wise orthogonal projection $\tilde x_0$ of $x_0$ to $N_0$ has an extension $\tilde x\in \L^0(\F_0;N)$ such that $x_0-\tilde x_0\in N_0^\perp$. Defining $\bar x^0:=x-\tilde x$, we have $f(\bar x^0)=f(x)$ everywhere in $S$ and $\bar x^0\in\cl\dom f$ $(\mu_t)_{t=0}^T$-almost everywhere. Repeating the argument for $t=1,\dots,T$, we arrive at an 
\[
\bar x^T\in \prod \L^0_t+ \sum_{t=0}^T \L^0(\F_t;N)
\]
with $f(\bar x^T)=f(x)$ and $\bar x^T_t\in N_t^\perp$ everywhere in $S$ for all $t$ and $\bar x^T\in\cl\dom f$ $(\mu_t)_{t=0}^T$-almost everywhere.

%By translation, we may assume that $m=0$ and $f(0)=0$. 
Let $(x^\nu)_{\nu=1}^\infty\subset\Phi$ such that $Ef(x^\nu)\le\inf\eqref{pf}+2^{-\nu}$. Since $f(x^\nu)$ is bounded in $\L^1$, Komlos' theorem gives the existence of a subsequence of convex combinations (still denoted by $(x^\nu)$) and $\beta\in \L^0$ such that $f(x^\nu)\le \beta$ almost surely. 

By the first paragraph, there exists $\bar x^\nu$ with $f(\bar x^\nu)=f(x^\nu)$ and $\bar x^\nu_t\in N_t^\perp$ everywhere in $S$ for all $t$, $\bar x^\nu\in\cl\dom f$ $(\mu_t)_{t=0}^T$-almost everywhere, and $\bar x^\nu\in \prod \L^0_t+ \sum_{t=0}^T \L^0(\F_t;N)$. Thus $\bar x^\nu\in\{x\in \L^0\mid x\in\Gamma\ \text{a.e.}\}$, where 
\[
\Gamma(s):=\{x \in\reals^{d(T+1)}\mid x_t\in N_t^\perp,\ f(x,s)\le \beta(s)\}.
\]
By Corollary~8.3.3 and Theorem~8.7 of \cite{roc70a}, the recession cone of $\Gamma(s)$ is given by $\Gamma^\infty(s)=\{x \mid x_t\in N_t^\perp, x\in N\}$. For $x\in \Gamma^\infty(s)$, we have $x_0\in N_t^\perp\cap N_t$, so $x_0=0$. Repeating the argument for $t=1,\dots ,T$, we get that  $x=0$ and so $\Gamma^\infty=\{0\}$ $\mu$-almost everywhere. By \cite[Theorem~8.4]{roc70a}, the sequence $(\bar x^\nu)$ is thus almost surely bounded. By Komlos' theorem, there exists a subsequence of convex combinations and $\bar x \in \L^0$ such that $(\bar x^\nu)\rightarrow \bar x$ $\mu$-almost everywhere. By Fatou's lemma, $\int f(\bar x)d\mu \le \liminf \int f(\bar x^\nu)d\mu\le\inf\eqref{pf}$.

Since $\prod_{t=0}^T\mu_t\ll\mu$, the sequence $(\bar x^\nu)$ converges $\prod_{t=0}^T\mu_t$-almost everywhere. By Lemma~\ref{lem:css} below, $\bar x^\nu=\sum_{t=0}^T(\tilde x^\nu)^t$ for some $\mu_t$-almost everywhere converging $(\tilde x^\nu)^t\in\L^0(\B(S_t),\mu_t)$, so $\bar x\in\cl\dom f$ $(\mu_t)_{t=0}^T$-almost everywhere. Let $\bar x^t$ be the limit of $(\tilde x^\nu)^t$. For every $t'$,
\[
\hat x^{t'}:=(-\bar x^{t'}_0,\dots,-\bar x^{t'}_{t'-1},\sum_{t\ne t'}\bar x^{t'}_t,-\bar x^{t'}_{t'+1},\dots,-\bar x^{t'}_T)
\]
belongs to $\L^0(N)$, so $x:=\bar x+ \sum_{t'=0}^T \hat x^{t'}$ satisfies $f(x)=f(\bar x)$. We also have that $x\in\Phi$ so $x$ is optimal. 
\end{proof}

\begin{remark}
The conclusion of Theorem~\ref{thm:exist} still holds if $f$ is coercive in the sense that $\{x\,|\,f^\infty(x,s)\le 0\}=\{0\}$. In fact, the proof then simplifies considerably. A more general condition that covers both the condition of Theorem~\ref{thm:exist} as well as the coercivity condition is that there is a subset $J$ of the indices $\{0,\ldots,T\}$ such that 
\[
\{x\,|\,f^\infty(x,s)\le 0\}=\{x\,|\,\sum_{t\in J}x_t=0,\ x_t=0\quad \forall t\notin J\}
\]
for all $s\in S$.
\end{remark}

\begin{remark}
When $f$ is given by \eqref{fgh}, we have
\[
f^\infty(x,s) = \sum_{t=0}^Tg_t^\infty(x_t,s_t)+h^\infty(-\sum_{t=0}^Tx_t,s),
\]
by \cite[Theorems~9.3 and 9.5]{roc70a} as soon as $f$ is proper.
\end{remark}

The following lemma was used in the proof of Theorem~\ref{thm:exist}. For $T=1$, more general results can be found e.g.~in \cite{fg97}; see also \cite{rt97}.

\begin{lemma}\label{lem:css}
If $(x^\nu)\subset\sum_{t=0}^T \L^0(\F_t,\prod_{t=0}^T\mu_t)$ converges $\prod_{t=0}^T\mu_t$-almost everywhere, then there exists $\mu_t$-almost everywhere converging sequences $((\tilde x^\nu)^t)\subset\L^0(\B(S_t),\mu_t)$ such that $\sum (\tilde x^\nu)^t=x^\nu$.% Moreover, if $\eta \in M_+(S)$ is such that $\eta_t\ll \mu_t$, then $(x^\nu)$ converges in $\L^0(\F,\eta)$.
\end{lemma}

\begin{proof}
The statement is clearly valid for $T=0$. We proceed by induction on $T$. Let $(\sum_{t=0}^Tx^\nu_t)$ be a converging sequence in $\sum_{t=0}^T \L^0(\F_t,\prod_{t=0}^T\mu_t)$ and let $A\subseteq S$ be the set where the convergence holds. %Passing to a subsequence if necessary, we may assume that it converges pointwise on a set $A$ of full measure. 

Let $A_T(s_T):=\{s^{T-1}\,|\,(s^{T-1},s_T)\in A\}$. Since $(\prod_{t=0}^T \mu_t)(A)=1$, we have $\mu_T(\bar A_T)=1$, where $\bar A_T:=\{s_T\,|\,\prod_{t=0}^{T-1} \mu_t(A_T(s_T))=1\}$. Let $\bar s_T\in\bar A_T$,
\begin{align*}
  (\tilde x^\nu)^T(s_T)&=(x^\nu)^T(s_T)-(x^\nu)^T(\bar s_T),\\
  (\tilde x^\nu)^{T-1}(s_{T-1}) &=(x^\nu)^{T-1}(s_{T-1})+(x^\nu)^T(\bar s_T),\\
  (\tilde x^\nu)^t&=x^\nu_t\qquad t=0,\ldots T-2,
\end{align*}
so that $\sum_t(x^\nu)^t = \sum_t(\tilde x^\nu)^t$. We have that $\sum_{t=0}^{T-1}(\tilde x^\nu)^t$ converges $\prod_{t=0}^{T-1}\mu_t$-almost everywhere and, by the induction hypothesis, there exist $\mu_t$-almost everywhere converging sequences $((\hat x^\nu)^t)\subset\L^0(\B(S_t),\mu_t)$ such that $\sum_{t=0}^{T-1}(\hat x^\nu)^t=\sum_{t=0}^{T-1}(\tilde x^\nu)^t$. We also get that
\[
\sum_{t=0}^T(x^\nu)^t -\sum_{t=0}^{T-1}(\tilde x^\nu)^t =(\tilde x^\nu)^T 
\]
converges $\mu_T$-almost everywhere. This completes the induction argument.
\end{proof}

\section{Applications to problems with fixed marginals}\label{sec:mt}

This section illustrates the results of the previous sections in the case of fixed marginals. More precisely, will assume throughout that $d=1$ (the measures are scalar-valued), and that $G_t$ and $H$ are given in terms of integral functionals with $g_t(x_t,s_t)=x_t$ for each $t$ and $h(\cdot,s)$ nondecreasing. In this case, $g_t^*(\cdot,s_t)=\delta_{\{1\}}$ and $\dom H^*\subset C^*_+$ so the assumptions of Lemma~\ref{lem:radon} are satisfied and problem \eqref{d} can be written as
\begin{equation}\label{mt}
\begin{alignedat}{2}
&\minimize\quad & & H^*(\lambda) \quad\ovr\ \lambda\in M\\
&\st\quad & &\lambda_t=\mu_t\quad t=0,\ldots,T,
\end{alignedat}
\end{equation}
while the relaxed primal \eqref{pr} problem becomes
\begin{equation}\label{mtpr}%\tag{$\text{PR}_\lambda$}
\minimize\quad \int_S\left[\sum_{t=0}^Tx_t + h(-\sum_{t=0}^T x_t)\right]d\mu \quad\ovr\quad x\in\Phi,
\end{equation}
where
\[
\Phi = \{x\in\prod_{t=0}^T\L^0_t\,|\, -\sum_{t=0}^T x_t\in D\quad (\mu_t)_{t=0}^T\text{-a.e.}\}.
\]

Combining Theorems~\ref{thm:dual}, \ref{thm:oc} and \ref{thm:exist} gives the following.

\begin{theorem}\label{thm:mt}
Assume that $h$ satisfies the assumptions of Theorem~\ref{thm:if}, that $h^*(v,\cdot)$ is $\mu$-integrable for $v\in\reals$ in a neighborhood of $1$ and that either $\mu=\prod_{t=0}^T\mu_t$ or $T=1$ and $\prod_{t=0}^T \mu_t\ll\mu$. Then the optima in \eqref{mt} and \eqref{mtpr} are attained, there is no duality gap and feasible solutions $x$ and $\lambda$ are optimal if and only if
\begin{align*}
d\lambda^a/d\mu &\in\partial h(-\sum_{t=0}^T x_t)\quad\mu\text{-a.e.},\\
d\lambda^s/d|\lambda^s| &\in N_D(-\sum_{t=0}^T x_t)\quad|\lambda^s|\text{-a.e.}
\end{align*}
If, in addition, $\mu=\prod_{t=0}^T\mu_t$, then $x_t\in\L^1_t$ for each feasible $x$ in \eqref{mtpr}.
\end{theorem}

\begin{proof}
Since $\dom G_t=C_t$ for all $t$ and $H$ is nondecreasing, Theorem~\ref{thm:dual} implies that the optimum in \eqref{mt} is attained and that there is no duality gap. To prove the attainment in \eqref{mtpr}, we apply Theorem~\ref{thm:exist} with
\[
f(x,s)=\sum_{t=0}^Tx_t + h(-\sum_{t=0}^Tx_t,s).
\]
Assumption~\ref{ass2} holds trivially since $\dom g_t=\reals^d$ for each $t$, so \eqref{mtpr} coincides with~\eqref{pf}. By the Fenchel inequality,
\begin{equation}\label{fenchel}
f(x,s) \ge (1-v) \sum_{t=0}^Tx_t - h^*(v,s), 
\end{equation}
so the integrability condition implies that the lower bound in Theorem~\ref{thm:exist} holds with $m(s)= h^*(1,s)$. This also gives
\[
f^\infty(x,s)\ge (1-v)\sum_{t=0}^Tx_t
\]
for $v$ in a neighborhood of $1$, so $f^\infty(x,s)\ge\epsilon|\sum_{t=0}^Tx_t|$ for some $\epsilon>0$. It follows that $f$ satisfies \eqref{eq:exist}. Thus, by Theorem~\ref{thm:exist}, the optimum in \eqref{mtpr} is attained. 

By Lemma~\ref{lem:ass1}, the integrability condition implies that the first part of Assumption~\ref{ass1} holds. If $T=1$, Lemma~\ref{lem:perp} implies that the second part of Assumption~\ref{ass1} is satisfied as well. If, on the other hand, $\mu=\prod_{t=0}^T\mu_t$, then, by Lemma~\ref{lem:ass1}, $x_t\in\L^1_t$ and the second part of Assumption~\ref{ass1} is again holds. The rest now follows from Theorem~\ref{thm:oc} by observing that, when $g_t(x_t,s_t)=x_t$, the condition $d\lambda_t/d\mu_t\in\partial g_t(x_t)$ simply means that $\lambda_t=\mu_t$.
\end{proof}

\subsection{Monge--Kantorovich problem}\label{ssec:linear}

Let $c$ be a measurable function on $S$ and let $h(u,s)=\delta_{(-\infty,c(s)]}(u)$. In this case, 
\[
h^*(v,s)=
\begin{cases}
c(s)v & \text{if $v\ge 0$},\\
+\infty & \text{otherwise}
\end{cases}
\]
and problem \eqref{mt} can be written as
\begin{equation}\label{mk}
\begin{alignedat}{2}
&\minimize\quad & & \int_S cd\lambda \quad\ovr\ \lambda\in M_+\\
&\st\quad & &\lambda_t=\mu_t\quad t=0,\ldots,T.
\end{alignedat}
\end{equation}
When $T=1$, we recover the classical Monge--Kantorovich mass transportation problem; see e.g.\ \cite{acbbv3}, \cite{vil9}, \cite{leo6}, \cite{rr98} and their references. On the other hand, if $S_t$ coincide for all $t$, problem \eqref{mk} can be interpreted as the problem of finding a stochastic process $X=(X_t)_{t=0}^T$ such that $X_t$ has distribution $\mu_t$ and the expectation of $c(X)$ is minimized. It should be noted that \eqref{mk} depends on $\mu$ only through its marginals $\mu_t$. Thus, we choose 
\[
\mu=\prod_{t=0}^T\mu_t.
\]

Problem~\eqref{mtpr} becomes
\begin{equation}\label{mkpr}
\begin{alignedat}{2}
&\minimize\quad & & \int_S\sum_{t=0}^T x_td\mu\quad\ovr\quad x\in\Phi,%\prod_{t=0}^T\L^1_t,\\
%  &\st & &\quad-\sum_{t=0}^T x_t\comp \pi_t\le c\quad\mu+|\lambda|\text{-a.e.}.
\end{alignedat}
\end{equation}
where
\[
\Phi = \{x\in\prod_{t=0}^T\L^0_t\,|\, -\sum_{t=0}^T x_t\le c\quad (\mu_t)_{t=0}^T\text{-a.e.}\}.
\]
Indeed, by Lemma~\ref{lem:mut}, $x\in\Phi$ implies $-\sum_{t=0}^Tx_t\in D$ $\mu$-almost everywhere so
\[
\int_S\left[\sum_{t=0}^Tx_t + h(-\sum_{t=0}^T x_t)\right]d\mu = \int_S\sum_{t=0}^T x_td\mu.
\]

\begin{theorem}
Assume that $c$ is lower semicontinuous and $\mu$-integrable with $c\ge K\psi$ for some $K\in\reals$. Then the optima in \eqref{mk} and \eqref{mkpr} are attained, there is no duality gap and feasible solutions $\lambda$ and $x$ are optimal if and only if
\[
%c+\sum_{t=0}^Tx_t \ge 0,\ \lambda\ge 0,\ 
\int_S\left(c+\sum_{t=0}^Tx_t\right)d\lambda=0.
\]
Moreover, if $x$ is feasible in \eqref{mkpr}, then $x_t\in\L^1_t$ so the objective of \eqref{mkpr} can be written as
\[
\int_S\sum_{t=0}^T x_td\mu = \sum_{t=0}^T\int_{S_t} x_td\mu_t.
\]
\end{theorem}

\begin{proof}
We now have $D(s)=\{u\in\reals\,|\,u\le c(s)\}$ which is inner semicontinuous if and only if $c$ is lower semicontinuous; see \cite[Example~1.2*]{mic56}. The lower bound on $c$ implies that $h$ satisfies the assumptions of Theorem~\ref{thm:if}. Since $c$ is $\mu$-integrable, all the conditions of Theorem~\ref{thm:mt} are satisfied. The form of the optimality conditions follows simply by observing that now, $\partial h=N_D$.
\end{proof}

Instead of the lower bound $c\ge K\psi$, \cite[Theorem~5.10]{vil9} assumes the existence of $c_t\in\L^1_t$ such that $c\ge\sum_tc_t$. However, if there is no $K\in\reals$ such that $c\ge K\psi$, then problem \eqref{p} is infeasible so the duality argument fails and, in particular, the first conclusion of \cite[Theorem~5.10]{vil9} does not hold. The function $c$ is integrable, in particular, if there exist $c_t\in\L^1_t$ such that $c\le\sum_tc_t$. This latter condition is assumed e.g.\ in \cite[Theorem~5.10]{vil9} in establishing the existence of solutions.

%\begin{remark}
%Instead of $c\ge 0$, Villani assumes $c\ge\sum_tc_t$ for upper semicontinuous integrable $c_t$ for dual existence. If $-\sum_tx_t\le c$, then $-\sum_t\tilde x_t\le\tilde c$, where $\tilde c= c-\sum_tc_t$ is nonnegative lsc and $\tilde x_t=x_t+c_t$ is upper semicontinuous??
%\end{remark}

\begin{remark}
Feasibility of an $x$ means that the inequality constraint holds on a product set $A^x=A^x_0\times\cdots\times A^x_T$, where $\mu_t(A^x_t)=1$. Thus, every dual feasible solution $\lambda$ satisfies
\[
\lambda((A^x)^c)\le\sum_{t=0}^T\lambda_t((A^x_t)^c) = \sum_{t=0}^T\mu_t((A^x_t)^c)=0.
\]
The optimality conditions thus imply that the optimal dual solutions $\lambda$ are supported by the sets 
\[
\Gamma_x:=\{s\in A^x\,|\,c(s)+\sum_{t=0}^Tx_t(s_t)\le 0\},
\]
where $x$ runs through optimal primal solutions. %Here $\Gamma$ is closed by lower semicontinuity $c$ and continuity of $u$ and $\sum x_t$.
 The sets $\Gamma_x$ are {\em $c$-monotone} in the sense that 
\[
\sum_{i=1}^n c(s^i_0,\dots,s^i_T)\le \sum_{i=1}^n c(s^{P_0(i)}_0,\dots,s^{P_T(i)}_T)
\]
for any $(s^i_0,\dots,s^i_T)\in \Gamma_x$, $i=1,\dots n$ and any permutations $P_t$ of the indices $i$. Indeed,
\[
\sum_{i=1}^n c(s^i_0,\dots,s^i_T) \le -\sum_{i=1}^n\sum_{t=0}^T x_t(s^i_t)= -\sum_{i=1}^n\sum_{t=0}^T x_t(s^{P_t(i)}_t) \le \sum_{i=1}^n c(s^{P_0(i)}_0,\dots,s^{P_T(i)}_T),
\]
where the last inequality follows from the feasibility of $x$ on $A^x$. This is a multivariate generalization of the $c$-cyclical monotonicity property studied e.g.\ in \cite{rr98} and \cite{vil9}. When $T=1$, it is known that a feasible $\lambda$ is optimal if it is concentrated on a $c$-monotone set. It would be natural to conjecture that this holds also for $T>1$.
\end{remark}

\subsection{Capacity constraints}\label{ssec:nonlinear}

Let $c$ and $\phi$ be nonnegative measurable functions on $S$ and let
\[
h(u,s)=%\begin{cases}
\phi(s)[u-c(s)]^+.% & \text{if $s\in\supp\mu$},\\
%\delta_{(-\infty,c(s)]}(u) & \text{if $s\notin\supp\mu$}.
%\end{cases}
%\begin{cases}
%  0 & \text{if $u\le c(s)$},\\
%  \phi(s)(u-c(s)) & \text{if $u\ge c(s)$}.
%\end{cases}
\]
We get
\[
h^*(v,s)=
%\begin{cases}
c(s)v + \delta_{[0,\phi(s)]}(v)%& \text{if $s\in\supp\mu$},\\
%c(s)v + \delta_{[0,+\infty)}(v)& \text{if $s\notin\supp\mu$},
%\end{cases}
\]
so problem \eqref{mt} can be written as
\begin{equation}\label{cc}
\begin{alignedat}{2}
&\minimize\quad & & \int_Scd\lambda \quad\ovr\ \lambda\in M_+\\
&\st\quad & &\lambda\ll\mu,\quad\frac{d\lambda}{d\mu}\le\phi,\quad\lambda_t=\mu_t\quad t=0,\ldots,T.
\end{alignedat}
\end{equation}
This models {\em capacity constraints} on the transport plan requiring $\lambda\le\phi\mu$, where the inequality is taken with respect to the natural order on $M$. Constrained variations of the Monge--Kantorovich problem are considered also in \cite[Chapter~7]{rr98}. What is called ``capacity constraints'' in \cite[Section~7.3]{rr98}, however, is different from the constraints of \eqref{cc}. In the case of finite $S$, problem~\eqref{cc} reduces to a network flow problem where the flow on each arc of the network is bounded from above by the value of $\phi$; see \cite{roc84} for a comprehensive study of linear and nonlinear network flow problems. 

Problem \eqref{mtpr} becomes
\begin{equation}\label{ccpr}
\begin{alignedat}{2}
&\minimize\quad & &\int_S\left[\sum_{t=0}^T x_t + \phi[\sum_{t=0}^Tx_t+c]^-\right]d\mu\quad\ovr\quad x\in\Phi,%\prod_{t=0}^T\L^1_t,\\
%  &\st & &\quad-\sum_{t=0}^T x_t\comp \pi_t\le c\quad\mu+|\lambda|\text{-a.e.}.
\end{alignedat}
\end{equation}
where
\[
\Phi = \prod_{t=0}^T\L^0_t. %\{x\in\prod_{t=0}^T\L^0_t\,|\, -\sum_{t=0}^T x_t\le c\quad (\mu_t)_{t=0}^T\text{-a.e.}\}
\]
Theorem~\ref{thm:mt} gives the following.

\begin{theorem}\label{thm:cc}
Assume that $\mu=\prod_{t=0}^T\mu_t$ and that $c$ and $\phi$ are $\mu$-integrable with $c\ge K\psi$ and $\phi\ge v$ for some $K\in\reals$ and $v>1$. Then the optima in \eqref{cc} and \eqref{ccpr} are attained, there is no duality gap and feasible solutions $\lambda$ and $x$ are optimal if and only if
\begin{align*}
d\lambda/d\mu &= 0\quad\text{if}\quad -\sum_{t=0}^T x_t < c, \\
d\lambda/d\mu &\in [0,\phi]\quad \text{if}\quad -\sum_{t=0}^T x_t = c, \\
d\lambda/d\mu &=\phi\quad \text{if}\quad -\sum_{t=0}^T x_t > c.
\end{align*}
Moreover, if $x$ is feasible in \eqref{ccpr}, then $x_t\in\L^1_t$.
\end{theorem}

In the case of finite $S$, the optimality conditions in Theorem~\ref{thm:cc} correspond to the classical complementary slackness conditions in constrained network optimization problems; see \cite{roc84}.

\subsection{Schr\"odinger problem}\label{ssec:sch}

We now return to the Schr\"odinger problem
%Given probability measures $R\in M$ and $\mu_t\in M_t$, the associated {\em Schr\"odinger problem} is the convex minimization problem
\begin{equation*}\label{sp}
\begin{alignedat}{2}
&\minimize\quad & & \int_S\ln(d\lambda/dR)d\lambda \quad\ovr\ \lambda\in M_+\\
&\st\quad & &\lambda\ll R,\quad\lambda_t=\mu_t\quad t=0,\ldots,T
\end{alignedat}
\end{equation*}
studied in Section~\ref{ssec:sch0}. We will derive optimality conditions and a dual problem under the assumption that there exists a feasible $\lambda$ equivalent to $R$. Denoting the feasible point by $\mu$ and $\phi:=d\mu/dR$, the problem can then be written as
\begin{equation*}
\begin{alignedat}{2}
&\minimize\quad & & \int_S\frac{d\lambda}{d\mu}\ln(\phi\frac{d\lambda}{d\mu})d\mu  \quad\ovr\ \lambda\in M_+\\
&\st\quad & &\lambda\ll \mu,\quad\lambda_t=\mu_t\quad t=0,\ldots,T.
\end{alignedat}
\end{equation*}
This fits the format of \eqref{mt} with $h(u,s)=\frac{e^u-1}{\phi(s)}$. Indeed, we have  
\[
h^*(v,s)=
\begin{cases}
  v\ln(\phi(s)v) - v +1/\phi(s) & \text{if $v >0$},\\
  0 &\text{if $v=0$},\\
  +\infty &\text{otherwise},
\end{cases}
\]
so that $(h^*)^\infty(\cdot,s)=\delta_{\{0\}}$ for all $s\in S$ and
\[
H^*(\lambda)=
\begin{cases}
  \int_S\frac{d\lambda}{d\mu}\ln(\phi\frac{d\lambda}{d\mu})d\mu & \text{if $\lambda\ll\mu$},\\
+\infty & \text{otherwise}.
\end{cases}
\]
The relaxed primal problem becomes
\begin{equation*}\label{sppr}
\begin{alignedat}{2}
&\minimize\quad & & \int_S\left[\sum_{t=0}^Tx_t + \frac{\exp(-\sum_{t=0}^Tx_t)-1}{\phi}\right]d\mu\quad\ovr\quad x\in\prod_{t=0}^T\L^0_t.
\end{alignedat}
\end{equation*}
Note that even when restricted to $x\in\prod C_t$, the objective is different from that in Theorem~\ref{thm:spd}.

Theorem~\ref{thm:mt} gives the following.

\begin{theorem}\label{thm:sp}
Assume that $T=1$, $\prod_{t=0}^T\mu_t\ll R$ and that \eqref{sp} admits a feasible solution equivalent to $R$. Then the optimum in \eqref{sp} is attained and the optimal solutions $\lambda$ are characterized by the existence of an $x\in\prod_{t=0}^T\L^0_t$ such that
\begin{align*}
d\lambda/dR &=\exp(-\sum_{t=0}^Tx_t)\quad R\text{-a.e.}.
\end{align*}
If $\prod_{t=0}^T\mu_t$ is feasible and equivalent to $R$, then the same conclusion holds for any $T$ and, moreover, $x_t\in\L^1_t$ for feasible $x$ in \eqref{sppr}.
\end{theorem}

%\begin{theorem}
%Assume that \eqref{sp} admits a feasible solution equivalent to $R$ and that $\prod_{t=0}^T\mu_t\ll R$. Then the optimum in \eqref{sp} and \eqref{sppr} are attained, there is no duality gap and feasible solutions $\lambda$ and $x$ are optimal if and only if
%\begin{align*}
%d\lambda/dR &=\exp(-\sum_{t=0}^Tx_t)\quad R\text{-a.e.}.
%\end{align*}
%If, in addition, $\prod_{t=0}^T\mu_t$ equivalent to $R$ and feasible in \eqref{sp}, then $x_t\in\L^1_t$ for each feasible $x$ in \eqref{sppr}.
%\end{theorem}

\begin{proof}
Since $\mu\approx R$, the condition $\prod_{t=0}^T\mu_t\ll R$ means that $\prod_{t=0}^T\mu_t\ll\mu$. The feasibility of $\mu$ in \eqref{sp} (and the definition of $\phi$) implies that the integrability condition in Theorem~\ref{thm:mt} is satisfied. It is clear that $h$ satisfies the other conditions as well. The optimality conditions mean that $\lambda\approx\mu$ and
\[
\frac{d\lambda}{d\mu}=\frac{\exp(-\sum_{t=0}^Tx_t)}{\phi}\quad\mu\text{-a.e.}
\]
which reduces to the one in the statement since $\mu\approx R$ and $\phi=d\mu/dR$.
\end{proof}

The necessity and sufficiency of the optimality condition in Theorem~\ref{thm:sp} was established for feasible solutions equivalent to $R$ in \cite[Theorem~3.43]{fg97} under the assumption that $R\ll\prod_{t=0}^T\mu_t$. Theorem~\ref{thm:sp} above gives the equivalence when $\prod_{t=0}^T\mu_t\ll R$ without assuming apriori the equivalence with $R$. 

The last statement of Theorem~\ref{thm:sp} seems new. An alternative condition for the integrability of $x_t$ is given in \cite[Proposition~1]{rt93}. Example~1 of \cite{rt93} shows that the integrability may fail without additional conditions.

% \begin{lemma}
% For $d=T=1$ and $\phi_0\ne\phi_1$,
% \[
%  \{x \in \N\mid \sum_t \phi_t x_t\le 0, \sum_t x_t \ge 0\} =\{0\}.
% \]
% In particular, \eqref{eq:assrec} is satisfied for $N=\{0\}$.
% \end{lemma}
% \begin{proof}
% Let $x$ belong to $\{x \in \N\mid \phi_0x_0+\phi_1x_1 \le 0, x_0+x_1 \ge 0\}$. Then 
% \[
% -x_1\le x_0 \le -\frac{\phi_1}{\phi_0}x_1
% \]
% everywhere, so
% \[
% \sup_{s_1}\{\phi_1x_1\} \le \phi_0\inf_{s_1}x_1.
% \] 
% for every $s_0$. Since $\phi_0$ and $\phi_1$ are densities with $\phi_0\ne \phi_1$, we have $\phi_1\not\ge \phi_0$ and $\phi\not\le\phi_0$.  Thus there exists $s_0$ and $s_1$ with $\phi_1(s_1)>\phi_0(s_0)$.

%  Thus $x_0$ has to be zero. Again, since $\phi_1$ is a density, it is nonnegative and thus $x_1=0$ as well.
% \end{proof}

\section{Appendix}

In this appendix we prove Theorem~\ref{thm:if}. The proof follows the arguments in \cite{per17} which in turn are based on those in \cite{roc71} and \cite{per14}. We reproduce the proofs here since we allow for unbounded scaling functions $\psi_t$ and we do not assume that $S$ is locally compact.

Let $h$ be a convex normal $\B(S)$-integrand on $\reals^d$, $\mu$ a nonnegative Radon measure on $S$ and let 
\[
I_h(u)=\int h(u)d\mu.
\]

\begin{theorem}\label{thm:app1}
If $I_h$ is finite and continuous at some point on $C$, then $I_h$ is lsc and $I_h^*$ is proper and given by
\[
I_h^*(\lambda)=\min_{\lambda'\in M}\{I_{h^*}(d\lambda'/d\mu)+\sigma_{\dom I_h}(\lambda-\lambda')\mid \lambda'\ll \mu\}.
\]
\end{theorem}

\begin{proof}
% By \cite[Theorems~1 and 2]{roc71}, $I_h$ is proper and lsc on $C_b$. Thus $H$ is proper and lsc as well.

% The properness of $I_h^*$ follows from ??  \cite[Theorem~2]{roc71}.

Defining the convex function $\bar I_h$ to $L^\infty$ by 
\[
\bar I_h(u)=\int h(u)d\mu,
\]
we have $I_h=\bar I_h\comp A$, where $A:C\rightarrow L^\infty(\mu)$ is the natural embedding. We equip $L^\infty$ with the essential supremum-norm. By \cite[Theorem~2]{roc71}, the continuity of $I_h$ at a point $\bar u$ implies that $\bar I_h$ is proper and continuous at $A \bar u$. Thus, by \cite[Theorem~19]{roc74},
\begin{align*}
I_h^*(\lambda)&=\inf_{\theta\in (L^\infty)^*}\{\bar I^*_h(\theta) \mid A^*\theta =\lambda\}.
\end{align*}
By \cite[Theorem~1]{roc71}, the conjugate of $\bar I_h$ on $(L^\infty)^*$ can be expressed in terms of the Yosida-Hewitt decomposition $\theta=\theta^a+\theta^s$ as
\[
\bar I_h^*(\theta)=I_{h^*}(d\theta^a/d\mu)+\sigma_{\dom \bar I_h}(\theta^s).
\]
We thus get
\begin{align}\label{eq:roc71a}
I_h^*(\lambda)&=\inf_{\theta\in (L^\infty)^*}\{I_{h^*}(d\theta^a/d\mu)+\sigma_{\dom \bar I_h}(\theta^s) \mid A^*(\theta^a+\theta^s) =\lambda\}.
\end{align}

It suffices to show that
\begin{align}\label{eq:roc71}
I_h^*(\lambda) &=  \inf_{\tilde\theta\in (L^\infty)^*,\theta^a\ll\mu} \{I_{h^*}(d\theta^a/d\mu) + \sigma_{\dom \bar I_h}(\tilde\theta)\mid A^*(\theta^a+\tilde\theta)=  \lambda\}.
\end{align}
Indeed, the formula in the statement follows by writing this as
\begin{align*}
I_h^*(\lambda)&=\inf_{\theta^a\ll\mu}\left\{I_{h^*}(d\theta^a/d\mu) + \inf_{\tilde \theta\in (L^\infty)^*}\{\sigma_{\dom \bar I_h}(\tilde\theta)\mid A^*\tilde \theta=\lambda-A^*\theta^a \}\right\},
\end{align*}
and using the expression
\[
\sigma_{\dom I_h}(\lambda-A^*\theta^a)=\inf_{\tilde \theta\in (L^\infty)^*}\{\sigma_{\dom \bar I_h}(\tilde \theta) \mid A^*\tilde\theta =\lambda-A^*\theta^a\},
\]
which is obtained by applying \cite[Theorem~19]{roc74} to the function $\delta_{\dom I_h}=\delta_{\dom \bar I_h}\comp A$.

To prove \eqref{eq:roc71}, let $\tilde\theta\in (L^\infty)^*$ such that $A^*(\theta^a+\tilde\theta)=  \lambda$. For any $u\in C$,
\begin{align*}
\langle u,\lambda \rangle -I_h(u) &=  \langle Au,\theta^a\rangle - \bar I_h(Au) + \langle u,A^*\tilde\theta\rangle,
\end{align*} 
so taking supremum over $u\in\dom I_h$ gives
\[
I_h^*(\lambda)\le I_{h^*}(d\theta^a/d\mu)+ \sigma_{\dom\bar I_h}(\tilde\theta).
\]
Minimizing over $\tilde\theta\in L^\infty(S)^*$ and $\theta^a\ll\mu$ such that $A^*(\theta^a+\tilde\theta)=  \lambda$ gives
\begin{align*}
I_h^*(\lambda) &\le  \inf_{\tilde\theta\in (L^\infty)^*,\theta^a\ll\mu} \{I_{h^*}(d\theta^a/d\mu) + \sigma_{\dom\bar I_h}(\tilde\theta)\mid A^*(\theta^a+\tilde\theta)=  \lambda\}.
\end{align*}
The reverse inequality follows by noting that if we restrict $\tilde\theta$ to be purely singular with respect to $\mu$, we obtain the right hand side of \eqref{eq:roc71a}. 
\end{proof}

\begin{theorem}\label{thm:app2}
If $D$ is isc and $C(D)\ne\emptyset$, then for each $\lambda\in M$,
\[
\sigma_{C_b(D)}(\lambda)=\int (h^*)^\infty(d\lambda/d|\lambda|)d|\lambda|.
\]
\end{theorem}

\begin{proof}
By Fenchel's inequality, 
\begin{align}\label{eq:fen1}
\langle y,\lambda\rangle \le \int \sigma_{D}(d\lambda/d|\lambda|)d|\lambda|
\end{align}
for every $y\in C_b(D)$, so it suffices to show
\[
\sup_{y\in C_b(D)} \langle y,\lambda\rangle\ge \int \sigma_{S}(d\lambda/d|\lambda|)d|\lambda|.
\]
We have, by \cite[Theorem 14.60]{rw98}, 
\begin{align*}
  \sup_{w\in L^\infty(\lambda;D)} \int wd\lambda =\int \sigma_{D}(d\lambda/d|\lambda|)d|\lambda|.
\end{align*}
Let $\tilde y\in C_b(D)$,
\[
\alpha< \int \sigma_{D}(d\lambda/d|\lambda|)d|\lambda|
\]
and $w\in L^\infty(\lambda;S)$ be such that $\int wd\lambda>\alpha$. By  Lusin's theorem \cite[Theorem 7.1.13]{bog7}, there is an open $\tilde O\subset S$ such that $\int_{\tilde O} (|\tilde y|+ |w_t|)d|\lambda|<\epsilon/2$, $\tilde O^C$ is compact and $w$ is continuous relative to $\tilde O^C$. The mapping
\[
\Gamma(s)=
\begin{cases}
	w(s)\quad&\text{if } s\in \tilde O^C\\
	D(s)\quad&\text{if } s\in \tilde O
\end{cases}
\]
is isc convex closed nonempty-valued so that, by \cite[Theorem 3.1''\!']{mic56}, there is a continuous $\hat y$ on $S$ with $\hat y=w$ on $\tilde O^C$ and $\hat y\in D$ everywhere. Since $\hat y$ is continuous and bounded on $\tilde O^C$ which is compact, there is an open $\hat O$ such that $\hat y$ is bounded on $\hat O$. Since $\hat O^C$ is a countable intersection of open sets, we may choose $\hat O$ in a way that $\int_{\hat O\backslash \tilde O^C} |\hat y_t|d|\lambda|<\epsilon/2$. 

Since $\hat O$ and $\tilde O$ form an open cover of $T$ and since $T$ is normal, there is, by \cite[Theorem~36.1]{mun0}, a continuous partition of unity $(\hat \alpha,\tilde\alpha)$ subordinate to $(\hat O,\tilde O)$. Defining $y:=\hat \alpha \hat y+\tilde \alpha \tilde y$, we have $y\in C_b(D)$ and
\begin{align*}
  \int yd\lambda &\ge \int_{\tilde O^C} wd\lambda-\int_{\hat O\backslash \tilde O^C} \hat\alpha |\hat y|d|\lambda|-\int_{\tilde O}\tilde\alpha |\tilde y|d|\lambda| \ge\int \alpha-\epsilon,
\end{align*}
which finishes the proof of necessity, since $\alpha<\int \sigma_{S}(d\lambda/d|\lambda|)d|\lambda|$ was arbitrary.

\end{proof}

\begin{theorem}\label{thm:ifb}
Assume that $D(s):=\dom h(\cdot,s)$ is isc, $\cl\dom H=C_b(D)$ and that $H$ is finite and continuous at some $u\in C_b$. Then $H$ is a proper convex lsc function and the restriction to $M$ of its conjugate is given by
\[
H^*(\lambda) = \int_S h^*(d\lambda^a/d\mu)d\mu + \int_S(h^*)^\infty(d\lambda^s/d|\lambda^s|)d|\lambda^s|,
\]
where $\lambda^s$ is the singular part of $\lambda\in M$ in its Lebesgue decomposition with respect to $\mu$.
%  Moreover, $\lambda\in\partial H(u)\cap M$ if and only if
% \begin{align*}
% d\lambda^a/d\mu &\in\partial h(u)\quad\mu\text{-a.e.}\\
% d\lambda^s/d|\lambda^s| &\in N_D(u)\quad|\lambda^s|\text{-a.e.}
% \end{align*}
If $\dom H=C_b$, then $\dom H^*$ is contained in the set of Borel-measures absolutely continuous w.r.t.\ $\mu$. %\subseteq \{\lambda\mid \lambda \text{ is a Borel-measure with } \lambda\ll\mu\}$. 
\end{theorem}

\begin{proof}

Since $\inte\dom I_h\cap C_b(D)\ne\emptyset$, \cite[Theorem 20]{roc74}  gives
\begin{align*}
 H^*(\lambda)=(I_h+\delta_{C_b(D)})^*(\lambda) &=\min_{\lambda''}\{ I_h^*(\lambda-\lambda'')+\sigma_{C_b(D)}(\lambda'')\}
\end{align*}
 Thus, by Theorem~\ref{thm:app1},
\begin{align*}
&H^*(\lambda)\\
&=\min_{\lambda''}\{\min_{\lambda'}\{I_{h^*}(d\lambda'/d\mu)+\sigma_{\dom I_h}(\lambda-\lambda'-\lambda'')\mid \lambda'\ll \mu\}+\sigma_{C_b(D)}(\lambda'')\}\\
&=\min_{\lambda'}\left\{I_{h^*}(d\lambda'/d\mu)+\min_{\lambda''}\{\sigma_{\dom I_h}(\lambda-\lambda'-\lambda'')+\sigma_{C_b(D)}(\lambda'')\}\midb \lambda'\ll \mu\right\}.
\end{align*}
Since $\inte\dom I_h\cap C_b(D)\ne\emptyset$, \cite[Theorem 20]{roc74} again gives 
\begin{align*}
H^*(\lambda) &=\min_{\lambda''}\{ \sigma_{\dom I_h}(\lambda-\lambda'')+\sigma_{C(D)}(\lambda'')\}.
\end{align*}
Since, by assumption, $C_b(D)=\cl \dom H =\cl(\dom I_h\cap C_b(D))$, the left side equals $\sigma_{C_b(D)}(\lambda)$. Thus
\begin{align}\label{eq:prf}
H^*(\lambda)&=\min_{\lambda'}\{I_{h^*}(d\lambda'/d\mu)+\sigma_{C_b(D)}(\lambda-\lambda')\mid \lambda'\ll \mu\}.
\end{align}
For $\lambda\in M$, Theorem~\ref{thm:app2} now gives
\begin{align*}
H^*(\lambda)&=\min_{\lambda'}\{\int h^*(d\lambda'/d\mu)d\mu+\int (h^*)^\infty(d(\lambda-\lambda')/d\mu)d\mu\}\\
&\quad+\int (h^*)^\infty(d(\lambda^s)/d|\lambda^s|)d|\lambda^s|,
\end{align*}
By \cite[Corollary 8.5.1]{roc70a}, the last minimum is attained at $d\lambda'/d\mu=d\lambda/d\mu$, so the last expression equals $J_{h^*}(\lambda)$.
%
% As to the subdifferential formulas, we have, for any $u\in\dom H$ and $\lambda\in M$, the Fenchel inequalities
% \begin{align*}
% h(u)+h^*(d\lambda^a/d\mu)&\ge u\cdot(d\lambda^a/d\mu)\quad\mu\text{-a.e.,}\\
% (h^*)^\infty(d\lambda^s/d|\lambda^s|) &\ge y\cdot (d\lambda^s/d|\lambda^s|)\quad|\lambda^s|\text{-a.e.},
% \end{align*}
% so $\lambda\in\partial H(u)$ if and only if $H(u)+J_{h^*}(\lambda)=\langle u,\lambda\rangle$ which is equivalent to having the Fenchel inequalities satisfied as equalities which in turn is equivalent to the given pointwise subdifferential conditions.
%

If $\dom H= C_b$, \eqref{eq:prf} implies the claim. 
\end{proof}

\begin{proof}[Proof of Theorem~\ref{thm:if}]
Defining $\tilde h(u,s)= h(\psi(s) u,s)$, $\tilde D(s)=\cl \dom \tilde h(s)$ and 
\[
\tilde H(u) := I_{\tilde h}(u)+\delta_{C_b(\tilde D)},
\]
on $C_b$, we get
\begin{align*}
H^*(\lambda) &=\sup_{u\in C}\{\langle u,\lambda\rangle -H(u)\}\\
&=\sup_{u\in C_b}\{ \langle u,\psi\lambda\rangle - H(\psi u)\}\\
&=\tilde H^*(\psi\lambda).
\end{align*}
Clearly, $\tilde D(s)=\{u \mid \psi(s)u \in \dom h(s)\}$. By \cite[Proposition 2.2]{mic56}, $D$ is isc if and only if $\tilde D$ is isc. It is thus clear that $H$ satisfies the assumptions of Theorem~\ref{thm:if} if and only if $\tilde H$ satisfies those of Theorem~\ref{thm:ifb}. Since $\tilde h^*(y,s)= h^*(y/\psi(s),s)$, an application of Theorem~\ref{thm:ifb} to $\tilde H^*(\psi\lambda)$ gives the expression for $H^*$ in the statement.

As to the subdifferential formulas, we have $\lambda\in\partial H(u)\cap M$ if and only if $H(u)+J_{h^*}(\lambda)=\langle u,\lambda\rangle$. For any $u\in\dom H$ and $\lambda\in M$, we have the Fenchel's inequalities
\begin{align*}
h(u)+h^*(d\lambda^a/d\mu)&\ge u\cdot(d\lambda^a/d\mu)\quad\mu\text{-a.e.,}\\
(h^*)^\infty(d\lambda^s/d|\lambda^s|) &\ge y\cdot (d\lambda^s/d|\lambda^s|)\quad|\lambda^s|\text{-a.e.},
\end{align*}
which hold as equalities if and only if $H(u)+J_{h^*}(\lambda)=\langle u,\lambda\rangle$. These equalities are equivalent to the given pointwise subdifferential conditions.

Since $\dom H^*=\{\lambda \in C^*\mid \psi\lambda \in \dom \tilde H^*\}$, the last claim follows from that of Theorem~\ref{thm:ifb}.
\end{proof}

% \begin{remark}
% If $S$ is locally compact, (or union of compact spaces), it seems that $I_h$ for $h=\delta_\Gamma$ and isc $\Gamma$, then, for every $y\in C_b(\Gamma)$, if $\partial I_h(y)\ne \{0\}$, then there is $\lambda\ne 0$ in $\partial I_h(y)$. And this should extend to $I_h+I_{h^1}$ for continuous $I_{h^1}$. 
% \end{remark}

% \begin{lemma}
% Let $X$ be a LCTVS that is paired with respect to two LCTVS $V$ and $\tilde V$ such that $\tilde V\subset V$.  Let $g$ be proper convex function on $X$ such that it is bounded above in a $\tau(X,\tilde V)$-neighborhood of the origin. Then the conjugate $g^*:V\rightarrow\ereals$ satisfies $\dom g^*\subset\tilde V$. Moreover, $g^*$ has inf-compact level sets.
% \end{lemma}
% \begin{proof}
% By translation, we may assume that $g(0)=0$. There exists $\sigma(\tilde V,X)$-compact $K$ and $\gamma>0$ such that $g\le\gamma$ on $K^\circ:=\{x\in X \mid \sigma_K(x)<1\}$. Thus $g\le \gamma+ \delta_{K^\circ}$ and $g^*(v)\ge \sigma_{K^\circ}-\gamma$. Here $\sigma_{K^\circ}(v)=j_K(v):=\inf_{\lambda>0} \{\lambda \mid  v\in \lambda K\}$, so $\dom \sigma_{K^\circ}\subset \tilde V$ and thus $\dom g^*\subset \tilde V$.
% \end{proof}

\bibliographystyle{plain}
\bibliography{sp}

\end{document}